\documentclass[11pt,letterpaper]{amsart}

\usepackage{amssymb,amsthm,amsmath,mathrsfs, mathtools, multirow}
\usepackage{thmtools}
\usepackage[dvipsnames]{xcolor}
\usepackage{bbm} % for \bbm{1}
\usepackage{geometry}     % Régler les marges
\usepackage{subcaption}

\usepackage[hypertexnames=false]{hyperref} %navigation
\hypersetup{colorlinks=true,allcolors=blue}
\usepackage{hypcap}
\usepackage[capitalise]{cleveref}
\usepackage{silence}
\usepackage{autonum}

\usepackage{imakeidx}
\makeindex[title=Index of terminology, columns=2] %index for terminology
\makeindex[name=symbols, title=Index of symbols, columns=3] %index for notation

\usepackage[utf8]{inputenc}
\usepackage{indentfirst}
\usepackage{soul}

\usepackage{comment}

\newtheorem{theorem}{Theorem}[section]

\newtheorem{lem}[theorem]{Lemma}
\newtheorem{prop}[theorem]{Proposition}

\theoremstyle{definition}
\newtheorem{defin}[theorem]{Definition}
\newtheorem{example}[theorem]{Example}

\theoremstyle{remark}

\newtheorem{rem}[theorem]{Remark}

\renewcommand{\rm}[1]{\mathrm{#1}}
\newcommand{\cal}[1]{\mathcal{#1}}
\newcommand{\bb}[1]{\mathbb{#1}}

\renewcommand{\frak}[1]{\mathfrak{#1}}

\newcommand{\R}{\bb R}

\newcommand{\Z}{\bb Z}

\newcommand{\N}{\bb N}

\newcommand{\calC}{\cal C}

\newcommand{\calO}{\cal O}

\newcommand{\diag}{\mathrm{diag}}

\newcommand{\SL}{\rm{SL}}
\newcommand{\SO}{\rm{SO}}
\newcommand{\GL}{\rm{GL}}

\DeclareMathOperator{\disc}{disc}

\renewcommand{\c}{c}

\title{Density of shapes of periodic tori in the cubic case 
}
\author{Nguyen-Thi Dang, Nihar Gargava, Jialun Li}

\date{}

\begin{document}
	\begin{abstract}
		Consider the compact orbits of the $\mathbb{R}^2$ action of the diagonal group on $\SL(3,\R)/\SL (3,\Z)$, the so-called periodic tori.
		For any periodic torus, the set of periods of the orbit forms a lattice in $\mathbb{R}^2$.
		Such a lattice, re-scaled  to covolume one, gives a shape point in $\SL (2,\R)/\SL(2,\Z)$.
		
		We prove that the shapes of all periodic tori are dense in $\SL (2,\R)/\SL(2,\Z)$. 
		This implies the density of shapes of the unit groups of totally real cubic orders.
	\end{abstract}
	\maketitle
	
\iffalse
	\renewcommand{\abstractname}{Resumé}
	\begin{abstract}
	    Considérons les orbites compactes de l'action du sous-groupe (isomorphe à $\mathbb{R}^2$) des matrices diagonales à coefficients positifs sur $\SL(3,\R)/\SL (3,\Z)$ : les tores périodiques.
	    Pour chaque tore périodique, l'ensemble des périodes de cette orbite forme un réseau de $\mathbb{R}^2$.
	    Un tel réseau, renormalisé pour être de coaire $1$, donne un point de forme dans $\SL (2,\R)/\SL(2,\Z)$.
	    
	    Nous prouvons que les formes des tores périodiques sont denses dans $\SL (2,\R)/\SL(2,\Z)$.
	    Cela implique la densité des formes des groupes des unités d'ordres cubiques totalement réels.
	\end{abstract}
\fi	
	
	\section{Introduction}
	
	In this paper, we investigate the shapes of periodic tori on $\SL(3,\mathbb{R})/\SL(3,\mathbb{Z})$.
	
	Denote by $A$ the subgroup of $\SL(3,\mathbb{R})$ of diagonal matrices with positive entries.
	Such a group is a locally compact topological group isomorphic to $\R^2$.
	A \emph{periodic torus} is a compact orbit for the left action of $A$ on $\SL(3,\R)/\SL(3,\Z)$.
    The set of periods of such an orbit  forms a discrete cocompact subgroup of $A$. 
    Any cocompact subgroup of $\R^2$ can be uniquely written as $g \Z^2$ for some $g \in \GL(2,\R)/\GL(2,\Z)$. 
    The \emph{shape} of a periodic torus is the projection of $g\Z^2$ in $\SL(2,\R)/\SL(2,\Z)$. 
    See \cref{de:shape} for a more precise definition.

	\textbf{Why study the shapes of periodic tori?}
    In the $\SL(2,\mathbb{R})$ case, the diagonal subgroup is one dimensional.
    Its left action on $\SL(2,\mathbb{R})/\SL(2,\mathbb{Z})$ is the geodesic flow on the unit tangent bundle of the modular curve $\mathbb{H}/\SL(2,\Z)$. 
    Periodic tori in this case are periodic orbits of the geodesic flow.
	Huber \cite{huberZurAnalytischenTheorie1959}, Margulis \cite{margulisCertainApplicationsErgodic1969} and Sarnak \cite{sarnak1980} among others have computed the asymptotic number of periodic geodesics of bounded length. 
	
%	{\nihar We can have a separate paragraph talking about the subtleties of this $M$. It looks a bit out of place under this title.}
%	In our case, periodic tori are actions of flat tori (cocompact quotients of $\R^2$) on $\SL(3,\R)/\SL(3,\Z)$.
	
	In our case, periodic tori are cocompact quotients of $A$. % which is isomorphic to $\R^2$.
{We endow $A$ with the Euclidean metric induced by the Killing form on the Lie algebra of $\SL(3,\R)$.
	Now cocompact quotients of $\R^2$ are topologically equivalent to genus 1 surfaces and are equipped with the induced quotient Euclidean metric.
	}
	Consequently, periodic tori may be ordered using different parameters;  volume or systole - length of the shortest geodesic - for example.
	Any such ordering will lead to a different counting problem i.e. to count the number of periodic tori for which the chosen parameter is bounded.
	The following are some orderings that were previously used for counting periodic tori :
	\begin{itemize}
	    \item The volume ordering, or geometric ordering  \cite{einsiedler_distribution_2009}.
	    \item The systole ordering, or dynamical ordering---the length of the shortest one-parameter periodic orbit \cite{spatzier83, deitmar,knieper2005uniqueness,dangEquidistributionCountingPeriodic2023,bonthonneau_srb_2021}.
%	    \item Ordering with respect to the systole in the space of Weyl chambers in the compact case \cite{}.
	\end{itemize}

	%{\nihar Should we add more volume-ordering results?}

	Our main motivation to study the shapes of tori is to understand the relation between volume orderings and systole orderings. For instance, when the shape lies within a compact region of the moduli space, the ratio between the systole and the square root of the volume cannot be too small. This relationship might aid in bridging the gap between counting results with the systole ordering to the volume ordering, offering a unified perspective on these two counting results.
	
	Another motivation comes from number theory.
	Let $M$ be the subgroup of diagonal matrices with entries in $\lbrace \pm 1 \rbrace$.
	Periodic tori in $M \backslash \SL(3,\R)/\SL(3,\Z)$ correspond to arithmetic objects, namely the unit groups of totally real cubic orders, as explained further in \cref{ss:orders}.
	Therefore the shape of a periodic torus is the shape of the unit group of a totally real cubic order. This bijection is the central theme of the seminal work of \cite{einsiedler_distribution_2011} where an ``arithmetic ordering'' of {tori} based on this connection is considered instead of geometric parameters like volume or systole. The present paper heavily relies on this bijection.

	\textbf{The space of Weyl chambers:}
	     The double coset $ M \backslash \SL(3,\R)/\SL(3,\Z)$ is called the space of Weyl chambers, which is a generalisation of the unit tangent bundle of a hyperbolic surface for higher-rank Lie groups (e.g., $\SL(d,\R)$ for $d\geq 3$), and the action of $A$ generalises the geodesic flow. Periodic tori also occur as compact orbits for $A \curvearrowright M \backslash \SL(3,\R)/\SL(3,\Z)$.

	Arithmetically, periodic tori in $M \backslash \SL(3,\R)/\SL(3,\Z)$ are in bijection with the unit groups of totally real cubic orders, as explained further in \cref{ss:orders}. 
	Hence we will mainly consider periodic tori in $M \backslash \SL(3,\R)/\SL(3,\Z)$.
	The case without quotienting by $M$, which requires some modifications, will be treated in \cref{ss:positive_units}.

	\textbf{Main result:}
	{The main result of the paper is the following \cref{co:ds}.}
	This solves a conjecture put forth in the paper \cite{davidShapesUnitLattices2017}, where the authors conjectured the following statement, motivated by numerical experiments 
	and questions of Margulis and Gromov.
	\begin{theorem}
	    \label{co:ds}
		The shapes of periodic tori in $ M\backslash \SL(3,\R)/\SL(3,\Z)$ are dense in $\SL(2,\R)/\SL(2,\Z)$.
%		As a consequence shapes of periodic tori in $M\backslash \SL(3,\R)/\SL(3,\Z)$ are also dense.
	\end{theorem}
Previously, Cusick \cite{cusickRegulatorSpectrumTotally1991} proved the closure of shapes contains an arc in the moduli space. 
{David-Shapira \cite{davidShapesUnitLattices2017} provided a description of additional curves that appear in this closure.} 
Both Cusick and David-Shapira rely on constructing a specific family of cubic polynomials to make cubic {orders}, which is also a key step in our proof (cf. \cref{eq:duke_poly}). 
%{\thi We add another step by considering orders $\{\Z + n \calO\}_{n \in \Z}$ obtained from the monogenic order $\calO = \Z[\alpha]$ generated by a root $\alpha$ of this polynomial.  }

	With \cref{co:ds}, we can now convince ourselves that any relationship between systole ordering and volume ordering will not be straightforward.
 Also, \cref{co:ds} tells us that the possible shapes of unit groups of orders of totally real cubic fields can be arbitrarily close to any given shape. This is of independent interest in number theory and is a partial result towards the conjecture stated in a MathOverflow discussion \cite{mathoverflow_disc}. %{\jialun It is an interesting discussion. I like their names, arithmetic ordering, geometric ordering and dynamical ordering. The author also mentioned a conjecture of Margulis about escape to infinity of dynamical ordering of shapes. }

One can also consider the periodic tori of $A \curvearrowright \SL(3,\R)/\SL(3,\Z)$ %instead of $M\backslash \SL(3,\R)/\SL(3,\Z)$ 
and study their shapes. 
{ 
It is a $4$-fold covering of the space of Weyl chambers and the action of $A$ is equivariant for this covering.
Therefore, periodic tori in $\SL(3,\R)/\SL(3,\Z)$ are covering of periodic tori in the space of Weyl chamber.
Since the covering may have some non-trivial index, the shapes are not necessarily preserved.
However, we can still prove density of shapes.
}
%This is only a small technicality and \cref{co:ds} also holds true when considering shapes of periodic tori without a quotient with respect to the finite group $M$ on the left. 
\begin{theorem}
    The shapes of periodic tori in $\SL(3,\R)/\SL(3,\Z)$ are dense in $\SL(2,\R)/\SL(2,\Z)$.
\end{theorem}
See \cref{ss:positive_units} for further details.

\begin{rem}
\begin{enumerate}
 \item There is a related notion of a shape studied by Bhargava-Harron \cite{bhargava2016equidistribution}. 
 Shape points of \cite{bhargava2016equidistribution} take values in $\SO(m)\backslash \SL(m,\R) / \SL(m,\Z)$, that is lattices up to scaling and isometry with $m$ being $\deg K-1$. 
 In \cite{yifrach2026equidistribution}, a related notion of a ``grid'' is introduced to consider the analogous construction in $\SL(m,R) / \SL(m,Z)$. % We don't make this distinction

Bhargava-Harron and Yifrach study the shape of the lattice coming from the additive group of an order in a number field whereas our work studies the lattices coming from the multiplicative group of units in the order, using Dirichlet's unit theorem. 
In the case of a totally real number field $K$ of degree $n$, both our notion and theirs associate a point in $\SL(n-1,\R)/\SL(n-1,\Z)$ to an order in $K$.   
% These authors previously proved equidistribution of shapes of maximal orders of number fields when ordered with respect to discriminant in degrees $n=3,4,5$.

The overloading of the same name is regrettable, but at the same time we fail to find a more appropriate name. We hope that a possible confusion can be quickly
remedied by the use of the adjectives ``multiplicative shapes'' vs ``additive shapes''.

\item 
{To our knowledge, there are no known tools to prove an equidistribution statement for our notion of shapes. 
However, as conjectured in \cite{mathoverflow_disc}, we expect similar equidistribution results to also hold in some natural orderings for totally real cubic orders.}
{Any input from the community in pursuing this direction is very welcome.}

\item {Corso--Rodriguez Hertz in a recent preprint \cite{corso2025} prove with different methods that the set of shapes is unbounded, which is a consequence of \cref{co:ds}.
They also present a comparison of the approaches.}
\end{enumerate} 
\end{rem}

\iffalse	
	\begin{rem}
		There is a third ordering of periodic tori, that was not mentioned above but is of considerable recent interest. 
		This can be called the ``arithmetic ordering''. Since each periodic torus arises from a totally real cubic order, one can 
		start arranging the periodic tori in ``packets'' of tori corresponding to the same cubic order.
		Then, \cite{einsiedler_distribution_2011} sorted the packets in growing order their discriminant 
		and proved the equidistribution of periodic tori on $\SL(3,\R)/\SL (3,\Z)$. %See \cref{re:ordre positive} for the reason why we module the sign group $M$ on the left.
	\end{rem}
	\fi
	
	\textbf{Sketch of the proof:} 
	By a classical observation (cf.  \cite{einsiedler_distribution_2009}), the periods of a periodic torus are in correspondence with the units of a totally real cubic order.
	The proof of the main theorem is based on a special family of cubic orders already studied by Cassels \cite{casselsProductInhomogeneousLinear1952} and Ankeny--Brauer--Chowla \cite{ankenyNoteClassNumbersAlgebraic1956}. The family of cubic orders considered here is
	\begin{equation}
	    \calO = \frac{\Z[X]}{\langle f(X) \rangle} , \text{ where } f(X) = X(X-a_1) (X-a_2) -1 ,
	\end{equation}
	for some integers $0<a_1<a_2$ in $\Z$.
	
	Using the ``best lower bound'' of Cusick \cite{cusickLower1983} between the regulator and the discriminant in the cubic case, we compute the generators of the unit group of these orders in \cref{pr:proposition_polynom-units-orders}.
	 As already observed by Cusick \cite{cusickRegulatorSpectrumTotally1991} and David-Shapira \cite{davidShapesUnitLattices2017}, this provides a family of shapes whose projections on $\mathbb{H}/\SL(2,\Z)$ are concentrated on the geodesic line between the square torus and the hexagonal torus in the moduli space of flat tori $\SL(2,\R)/\SL(2,\Z)$. (See \cref{fig:sub1}).
	 The proof of the compactness of these shapes is given in \cref{ss:cubic-field}. {This makes it seem like an obstruction to show density of shapes, however the final proof of \cref{co:ds} uses this compactness as an essential step instead of an obstruction.}
	 
	 \begin{figure}[h!]
\centering
\begin{subfigure}{.45\textwidth}
  \centering
  \includegraphics[width=.6\linewidth]{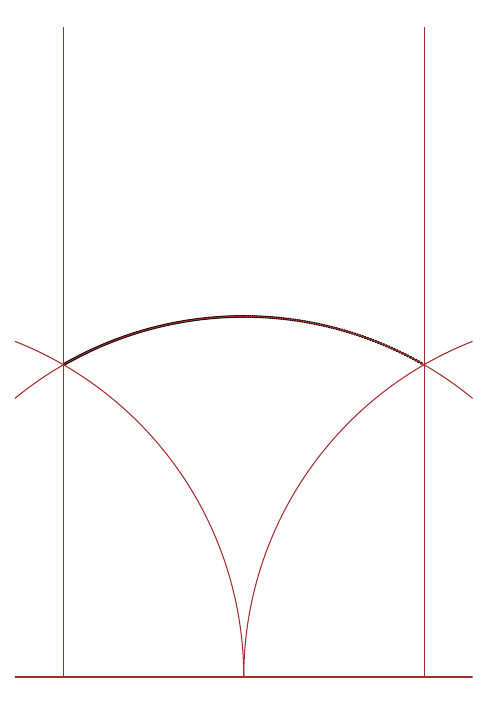}
  \caption{Shape points from \cref{le:where_is_shape} with\\ $a_1\in\{5^j,1\leq j\leq 199 \}$ and
  $a_2=5^{200}$. %\footnote{This figure shows that the shape points given by these polynomials can approximate the whole arc \cref{prop-bounded-shape}} 
  }
  \label{fig:sub1}
\end{subfigure}%
\begin{subfigure}{.45\textwidth}
  \centering
  \includegraphics[width=.6\linewidth]{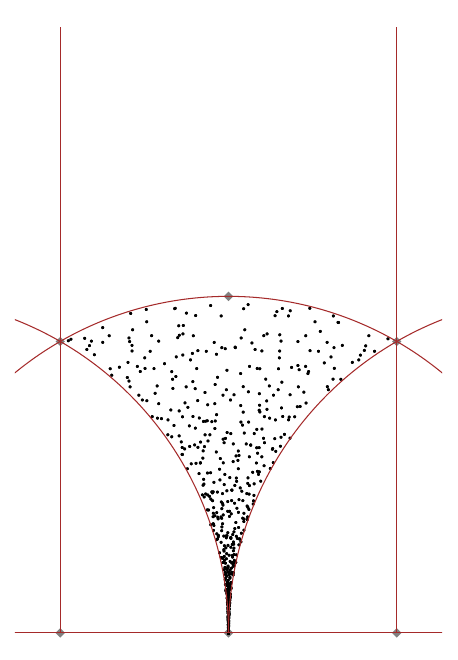}
  \caption{Shape points from \cref{defin-Lcde} $(p_1,p_2,p_3)=(2,3,5)$ and $c,d,r \in [1,10]$.}
  \label{fig:sub2}
\end{subfigure}
\caption{Figures of shape points in a fundamental domain of $\mathbb{H}/\SL(2,\Z)$.}

%\caption[Caption for LOF]{Real caption\protect\footnotemark}
\label{fig:test}
\end{figure}
%\footnotetext{blah blah blah}
%{\jialun{Maybe we can explain what are the points later and in more details}}
	 
	 {  
	 In order to obtain more shapes, one can consider finite index subrings / suborders of the above cubic orders. Namely, for an integer $n \in \mathbb{Z}_{\geq 1}$ and any cubic order from the above family $\calO$, one can try to work with $\calO_n = \mathbb{Z} + n \calO$. Overall, $\calO_n^{\times}$ depends highly on the prime factorization of $n$. 
	 Assuming that $n= p_1^{r_1} \cdots p_k^{r_k}$ is the prime factorization for $n$, one can separately analyse the orders $\calO_i = \mathbb{Z} + p_i^{r_i} \calO$ to get $\calO_{n} = \bigcap_i \calO_i$ (see \cref{le:local_behaviour}). To control $\calO_n^{\times}$, we try to work with a fixed set of primes $p_1,\ldots,p_k$ and increase the exponents $r_1,\ldots,r_k$. When $\calO$ is fixed and $r_1,\ldots,r_k$ are varied, it seems to be a difficult problem in $p$-adic Diophantine approximation to show the density of the shapes of $(\calO_n)^{\times}$.
	 }
	 % prove in \cref{prop:suborder} that the units of the suborder $\Z+p_1^c p_2^d p_3^r\calO$ may be computed using generators of $\calO^\times$ and three congruence equations satisfied by the powers of these generators, our so-called $3 \times 2$-matrix congruence equation.

{What was needed to prove \cref{co:ds} was to impose, for every $N\in \N$, local conditions on the coefficients of the polynomial $f_N(X)$ up to a depth $N$ (that is, congruence conditions modulo $p_1^N p_2^N p_3^N$) as we do in \cref{ss:proof_of_density}. 
The outcome of these congruence conditions is twofold. Firstly, in \cref{prop:suborder}, we are able to explicitly compute the unit group $(\calO_{p_1^cp_2^dp_3^r})^\times$, and hence its shape, for $0\leq c,d,r\leq N$ whenever these congruence conditions are satisfied. Secondly, if $\calO^{\times}$ could be fixed, this gives roughly $\sim N^3$ shape points that are finite index sublattices
of $\calO^{\times}$.
We prove that the closure of such shape points satisfy an invariance under the diagonal group and contain a horospherical piece. 
By using techniques similar to the ``banana trick'' argument due to Margulis (see  \cite{kleinbock1996bounded, shah_1996}), this implies the density of shapes when we restrict our argument to a stationary sequence of shape points of $\calO^{\times}$. This is where the proof uses the compactness of the shapes $\calO^{\times}$ when the polynomial $f$ is allowed to vary by changing $a_1,a_2$ in \cref{eq:duke_poly} satisfying our local conditions modulo $(p_1p_2p_3)^N$. The exact choice of $p_1,p_2,p_3$ is not critical for our proof and we work with the primes $2,3,5$ (but at least 3 primes seem to be necessary).

% \cref{theorm_density}, we prove that for the square $2$-lattice $\Z^2$, extracting sublattices using such equations will approximate all shapes in $\SL(2,\mathbb{R})/\SL(2,\mathbb{Z})$. 

}

{ To sum up, the family of orders that allows us to prove the density is $\Z+2^c 3^d 5^r\calO$, where $c,d,r \in \N \cap [1,N]$ and $\calO = \Z[X]/\langle f_N(X) \rangle$ where $f_N$ ranges over a family of polynomials defined in \cref{ss:density}, as $N$ goes to infinity. 
}
\subsection{Some side remarks}

\begin{enumerate}
    \item For the higher dimensional case of $\SL(n,\R)/\SL(n,\Z)$ with $n\geq 4$, there exist similar inequalities as Cusick's inequality in cubic case such as \cite{remakUberGrossenbeziehungenZwischen1952} and \cite{silvermanInequalityRelatingRegulator1984}. But the bound is not very good. With the same method, we can only locate the lattice $\calO^\times$ up to finite index. Therefore, our method cannot be directly generalized to higher dimension, although we can obtain some unboundedness result of shapes. 
    
    \item 
	It is interesting that we can add non-Archimedean local conditions to obtain a collection of shapes that admit diagonal and horospherical invariance, which is the key to the density. 
	We believe that these techniques indicate a deeper arithmetic structure of shapes of orders which might shed more light on the distribution of shapes of units of orders in the future.
	
	\item 
		
	{ 
	Our work was born out of efforts to understand the effect of considering Hecke neighbours of periodic tori. We owe it to the work of \cite{solan2023tori} as an inspiration
	to investigate this direction. They use Hecke correspondence on periodic tori to construct interesting limit measures of diagonally invariant probability measures.
	
	However, our main direction of investigation is the space of shapes ($\SL(2,\R)/\SL(2,\Z)$ in our case) and not the space of the periodic tori themselves ($\SL(3,\R)/\SL(3,\Z)$ in our case). The exact same technique therefore does not apply. Indeed, we need to use suborders $\Z+2^c 3^d 5^r\calO$ with at least three different primes to get our result and a single prime does not suffice. Another main difference is that we make no use of equidistribution of Hecke points \cite{clozelEQUIDISTRIBUTIONPOINTSHECKE2003}. 
	
	We reemphasize that the idea of using the polynomials $X(X-a_1)(X-a_2)-1$ is itself very classical and goes back to Cassels, Shapira and many other authors. 
	}

\end{enumerate}

	\section*{Acknowledgements}
	
	We would like to thank 
	Nalini Anantharaman for a research visit with helpful discussions at the 
	Institut de Recherche Mathématique Avancée (Strasbourg), Samuel Lelièvre, Vincent Tugayé for help with numerical simulations and Menny Aka, Yves Benoist, Fanny Kassel, François Labourie, Fran\c{c}ois Maucourant for helpful conversations during this project. {We also wish to thank the secretarial staff of the Centre de Math\'ematiques Laurent Schwartz and Laboratoire de Math\'ematiques d'Orsay for their outstanding assistance in organizing the mini-workshop ``Diagonal actions in the space of lattices" that allowed the joint work to begin.} %{\nihar Better to write non-abbreviated CMLS and LMO}
	
	While working on this project, N.G. received funding support from the ERC Grant\footnote{
Funded by the European Union. Views and opinions expressed are however those of the author(s) only and do not necessarily reflect those of the European Union or the European Research Council Executive Agency. Neither the European Union nor the granting authority can be held responsible for them.
}
	titled ``Integrating Spectral and Geometric data on Moduli Space''. J.L. would like to acknowledge the support of a CNRS starting grant.
	N.-T. D. also acknowledges support from Projet ANR-23-CE40-0001.
	
	This first version of this paper predates AI-assisted mathematics. All the main ideas of this paper are based on human insight. However, the current version is now accompanied by an auto-formalization written in \texttt{Lean} that certifies all the arguments of this paper \cite{gargava_density_of_shapes}. This auto-formalization was then used to correct some very minor mathematical details in the paper.
	
	\section{Background}
	
	\subsection{Number theoretic preliminaries}
	\label{ss:nt_prelims}
	
	Let $(K,\sigma)$ be a totally real cubic field with ordered real embeddings, where $K$ denotes the cubic field and  $\sigma:=(\sigma_i)_{1\leq i\leq 3}$ is a fixed ordering of the three real embeddings $\sigma_1,\ \sigma_2,\ \sigma_3$ of $K$.
	Let $\calO$ be an order in $K$, which is a finite index subring of the ring of integers $\calO_K$. 
	The map $\sigma$ gives an embedding of $\calO$ into $\R^3$ as a rank-three free abelian group, that is, a rank-three lattice. 
	The squared covolume of $\sigma(\calO)$ in $\R^3$ is the \emph{discriminant} of the order $\calO$ denoted by $\disc(\mathcal{O})$.
	
	\begin{example}
		\label{re:example_of_order}
		%{\thi Do we want the polynomial to have only real roots?}
		Fix a monic irreducible polynomial $f(X) \in \mathbb{Z}[X]$ of degree 3 with three real roots. Then $K = \mathbb{Q}[X]/\langle f(X) \rangle$ is a cubic field in which $\calO = \mathbb{Z}[X]/\langle f(X) \rangle$ is an order. The discriminant of this order is given by 
		\begin{equation}
			\disc(\calO) = \prod_{i<j\in\{1,2,3\} } (\alpha_i - \alpha_j)^{2},
		\end{equation}
		where $\alpha_1,\alpha_2,\alpha_3 \in \overline{\mathbb{Q}}$ are the three roots of $f$.
		
	\end{example}

	Let $\calO^\times$ be the subset of invertible elements in $\calO$, which is a multiplicative group. We recall an important theorem concerning the structure of $\calO^\times$.
	\begin{theorem}
		\label{th:dut}
		{\textbf (Dirichlet's unit theorem)}
		
		Given an order $\calO \subseteq \calO_K$ for a totally real cubic field $K$, the group $\calO^\times$ is isomorphic to $\Z^2\times \{\pm 1\}$.  That is, there exist two generators $\alpha,\beta \in \calO^{\times}$ such that 
		\begin{equation}
			\calO^{\times} = \langle \alpha,\beta , \pm 1 \rangle. 
		\end{equation}
	\end{theorem}
	For any $x\in\calO^\times$, we define the map $\psi$ by
	\begin{equation}
		\psi(x)=
		\begin{pmatrix}
			\log|\sigma_1(x)| &  & (0) \\
			& \log|\sigma_2(x)| & \\
			(0) & & \log|\sigma_3(x)|
		\end{pmatrix}.
		\label{eq:defi_of_psi}
	\end{equation}
	
	Since $x\in\calO^\times$, we know that the product $\sigma_1(x)\sigma_2(x)\sigma_3(x)=\pm 1$, hence the image of $\psi$ is inside the hyperplane $\frak a=\{\diag(x_1,x_2,x_3),\ x_i\in\R,\ x_1+x_2+x_3=0 \}$. 
	We know by Dirichlet's unit theorem that the image $\sigma(\calO^\times)$ is a rank-$2$ $\mathbb{Z}$-lattice under multiplication and hence maps to a $2$-lattice in $\frak a$ via $\psi$. 
	We fix one identification of $\frak a$ with $\R^2$ from now on. We denote the image of $\psi(\calO^\times)$ by $\Lambda(\calO)$, a $2$-lattice in $\R^2$. The covolume of the lattice $\psi(\calO^\times)$ in $\frak a$ is $\sqrt{3}$ times the \emph{regulator} $R(\calO)$ of the order $\calO$. 
	
	\begin{defin}[Definition of shapes]
	\label{de:shape}
	{For any $2$-lattice $\Lambda$ in $\R^2$, there exists an element $g\in \SL(2,\R)$ such that $g\, \mathbb{Z}^{2}=  c \cdot \Lambda$ for some $c > 0$.
	We denote by $[\Lambda]:=g\, \SL(2,\mathbb{Z})$ its shape, which is in the space of covolume-one $2$-lattices $\SL(2,\R)/\SL(2,\Z)$.}

		{The \emph{shape} of $\psi(\calO^\times)$ is defined to be $[\Lambda(\calO, \sigma)]$.
		In the rest of the paper, for every totally real cubic field we consider, the choice of ordered real embeddings $\sigma$ will be fixed.
		Abusing notation, we omit $\sigma$ and denote the shape point by
		$[\Lambda(\calO)]$.} 
	\end{defin}
	
	We state a result of Cusick which relates the regulator and discriminant of orders in cubic fields.\footnote{The result is stated for maximal orders, but the proof works verbatim for general orders. See also \cite[Theorem 5.8]{conradUnit}.} We will use this theorem in \cref{pr:proposition_polynom-units-orders}.
	\begin{theorem}
		\label{th:cusick}
		{\textbf (Cusick, \cite{cusickLower1983})}
		
		For any order $\mathcal{O}$ in a totally real cubic number field $K$, one has the inequality given by 
		\begin{equation}
			R(\mathcal{O}) \geq  \frac{1}{16}  \log^2 (\disc  (\mathcal{O}) /4).
		\end{equation}
	\end{theorem}
	\subsection{Units of orders and periods of periodic tori}
	\label{ss:orders}
Using the ring structure of the order $\calO$, we can realize $\calO^\times$ as a commutative subgroup of $\GL_3(\Z)$. 
	Similarly as above, the map $\Psi$ is defined for every $x\in \calO$ by
	$$\Psi (x) = \begin{pmatrix}
		\sigma_1(x) &  & (0) \\
		& \sigma_2(x) & \\
		(0) & & \sigma_3(x)
	\end{pmatrix}
	. $$
	Note that $\psi = \log \vert \Psi \vert$. 
	The action of $\calO^\times$ on $\sigma(\calO)$ is given by $\Psi(x)\sigma(y)=\sigma(xy)$. 
	This action is linear and preserves the $3$-lattice $\sigma(\calO)$. It may not necessarily be orientation preserving.

For every oriented basis $y = (y_1,y_2,y_3)$ of the $\Z$-lattice $\calO$, we set
\begin{equation}\label{eq-def-h_y}
h_y= \frac{1}{\rm{disc} (\calO)^{1/6} } \big(\sigma(y_1), \sigma(y_2),\sigma(y_3)\big)
\end{equation}
Note that $\det\big(\sigma(y_1), \sigma(y_2),\sigma(y_3) \big)=\sqrt{\rm{disc}(\calO)}$. As a consequence, $h_y \in \SL_3(\R)$.
We construct a map $\gamma$ from $\calO^\times$ to $\GL_3(\Z)$ as follows. 
For all $x\in \calO^\times$, we set
	\[ \gamma(x)=  h_y^{-1} \Psi(x) h_y. \]
A different choice of oriented basis only changes $\gamma$ up to conjugation in $\GL_3(\Z)$. 
	
	Let $A$ be the diagonal group $\{\diag(a_1,a_2,a_3):\ a_i>0 \;\; \forall i ,\ a_1a_2a_3=1\}$ and $M$ be the sign group $\{\diag(\varepsilon_1,\varepsilon_2,\varepsilon_3):\ \varepsilon_i=\pm 1 \;\; \forall i, \; \varepsilon_1\varepsilon_2\varepsilon_3=1\}$.
	\begin{lem}\label{lem:order-gamma}
		For all orders $\calO$ and every choice of basis $y=(y_1,y_2,y_3) \in \calO^3$ and for the above notations,
		$$h_y^{-1}(MA)h_y\cap \SL_3(\Z)=\gamma(\calO^\times)\cap \SL(3,\R).$$
	\end{lem}
	\begin{proof}
		From the construction, we know that $\gamma(\calO^\times)\cap\SL_3(\R)$ is a subset of $h_y^{-1}(MA)h_y\cap \SL_3(\Z) $.
		
		For the other direction, notice that $ \mathrm{disc}(\mathcal{O})^{-\frac{1}{6}}\sigma(\calO)=h_y\Z^3$.
		Let $\beta$ be any element in $\SL_3(\Z)\cap h_y^{-1}(MA)h_y$.
		Since $\SL(3,\Z)$ preserves the $3$-lattice $\Z^3$, on one hand,
		the matrix $h_y\beta h_y^{-1}$ preserves the $3$-lattice $\sigma(\calO)$.
		On the other hand, it is in $AM$, therefore diagonal.
		We denote its diagonal coefficients by $(\ell_i)_{i=1}^3$ and, abusing notation, write $\ell = h_y\beta h_y^{-1}$.
		Since $\calO$ is a subring, it contains $\pm 1$. 
		In particular $\ell \sigma(1)\in\sigma(\calO)$. 
		Hence there exists an element $x\in\calO$ such that $\ell_i =\sigma_i(x)$ for all $1 \leq i \leq 3$. 
		Consequently $h_y\beta h_y^{-1}=\Psi(x)$. 
		
		The same argument also works for $\beta^{-1}$. Therefore $h_y\beta h_y^{-1}\in \Psi(\calO^\times) $, which implies $\beta\in \gamma(\calO^\times)$.
	\end{proof}	
%		Consider the diagonal orbit $F=MAg\SL_3(\Z)$ on $M\backslash\SL_3(\R)/\SL_3(\Z)$ with $g\in\SL_3(\R)$. Let $\frak a=\{\diag(a_1,a_2,a_3):\ a_i\in\R,\ a_1+a_2+a_3=0 \}$ which is the Lie algebra of $A$. The set of periods of $F$ is defined by
		Consider the diagonal orbit $F=MAg\SL_3(\Z)$ on $M\backslash\SL_3(\R)/\SL_3(\Z)$ with $g\in\SL_3(\R)$. The Lie algebra of $A$ is
		\begin{equation}
			\label{eq:lie-algebra-A}
			\frak a=\{\diag(a_1,a_2,a_3):\ a_i\in\R,\ a_1+a_2+a_3=0 \}.
		\end{equation}
		The set of periods of $F$ is defined by
	\[\mathrm{P}(F)=\{v\in \frak a:\ \exp(v)z=z \; , \; \forall z \in F \}.  \]
	A diagonal orbit $F$ is called a periodic torus if $\mathrm{P}(F)$ has finite covolume in $\frak a$.
	
	\begin{lem}
		The orbit $ F_\calO=MAh_y\SL_3(\Z)$ is periodic and its set of periods is given by $\psi(\calO^\times)$. 
		As a consequence, the shape of the periods of $F_\calO$ is exactly $[\Lambda(\calO)]$.
	\end{lem}
	\begin{proof}
		The orbit $F_\calO$ does not depend on the choice of $h_y$, since different choices of $y$ differ by right multiplication by an element of $\SL(3,\Z)$.
		A period $v$ of $F_\calO$ is given by $m\exp(v)h_y=h_y\gamma$ for some $m\in M$ and $\gamma\in\SL(3,\Z)$. Due to \cref{lem:order-gamma}, we obtain $v\in \psi(\calO^\times)$. 
		
		For each $x\in\calO^\times$ with norm $1$, we can also find $v$ by \cref{lem:order-gamma}. For each element $x$ in $\calO^\times$ with norm $-1$, the element $-x$ has norm $1$ and $\psi(x)=\psi(-x)$.
	\end{proof}
\begin{rem}\label{re:order positive}
\begin{enumerate}
\noindent
    \item Most suborders of a given order will give, up to homothety, a different oriented basis for the $3$-lattice.
    As a consequence, they will correspond to another periodic torus, i.e. an abelian subgroup of $\SL(3,\Z)$ of dimension $2$.
	
	\item The above construction starts from a totally real cubic order and yields a periodic torus. 
	The correspondence may be generalised to orders of totally real polynomials of higher degree 
	(Cf. \cite{einsiedler_distribution_2009}).

	\item Two different periodic tori can have the same set of periods, hence the same shape.

	% More precisely, there are examples of totally real cubic order $\calO$, generators $y'=(y_1',y_2',y_3')$ of a $3$-sublattice of $\calO$ (not homothetic to $\calO$) such that $h_{y'}^{-1} M \Psi(\calO^\times)h_{y'} \subset \SL(3,\Z)$ and $h_{y}^{-1} M \Psi(\calO^\times)h_{y}$ are different periodic tori with the same set of periods. 
	{Indeed, consider any two ideals $\mathcal{I},\mathcal{J} \subseteq \mathcal{O}_K$ of the maximal order $\calO_K \subseteq K$ of a totally real cubic field $K$. When viewed as points in $M \backslash \SL(3,\R)/\SL(3,\Z)$, $\mathcal{I},\mathcal{J}$ are in the same $A$-orbit if and only if $\mathcal{I} = \alpha \mathcal{J}$ for some element $\alpha \in K^{\times}$. Hence, if $\mathcal{I},\mathcal{J}$ are in two different ideal classes, they are not in the same diagonal orbit. However, one can check that their shapes are the same since
	$$\{\alpha \in K^{\times} \mid \alpha \mathcal{I} = \mathcal{I}\} = \mathcal{O}_K^{\times} = \{\alpha \in K^{\times} \mid \alpha \mathcal{J} = \mathcal{J}\}.$$
	}
		
	\item 
	\label{reit:positive_units}
	We consider $M\backslash \SL(3,\R)/\SL (3,\Z)$ instead of $\SL(3,\R)/\SL (3,\Z)$ for the following reason: the periods of periodic tori in $\SL(3,\R)/\SL (3,\Z)$ correspond to totally positive units, that is $\calO^{\times,+}=\{x\in\calO^\times:\ \sigma_j(x)>0 \text{ for }j=1,2,3 \}$. From a number-theoretic perspective, the shape of $\calO^\times$ is more natural, which aligns with the setting of $M\backslash \SL(3,\R)/\SL (3,\Z) $.
	
	{For completion, we will show the density of shapes of periodic tori in $\SL(3,\R)/\SL(3,\Z)$ in \cref{ss:positive_units}.}
	
\end{enumerate}	
	\end{rem}
	
	%\subsection{Shape of the set of periods}

	\section{Shapes of periodic tori for a family of cubic fields}
	\label{ss:cubic-field}
	Let $3 \leq a_1 < a_2$ be integers.
	Consider the cubic polynomial 
	\begin{equation}
	\label{eq:duke_poly}
	   f(X)=X(X-a_1)(X-a_2)-1.
	\end{equation}
	
%Denote by $\alpha$ the smallest root of $f$ and let $K=\mathbb{Q}(\alpha)$. Consider the order $\mathcal{O}=\Z[\alpha]$. 
Denote by $\theta_0$ the smallest root of $f$ and let $K=\mathbb{Q}(\theta_0)$. Consider the order $\mathcal{O}=\Z[\theta_0]$. 
%The point of this polynomial is that one can control where the roots are, as shown in the upcoming Lemma. 
The purpose of introducing this polynomial is to control where the roots are, as shown in the upcoming lemma. 
Polynomials of this kind are well studied; see, for example, \cite{casselsProductInhomogeneousLinear1952} and \cite{ankenyNoteClassNumbersAlgebraic1956}.
	
	\begin{lem}
		\label{le:control_the_roots}
%		Let $\alpha < \alpha_1 < \alpha_2$ be the three roots of $f(X)$. Then, 
The polynomial $f(X)$ has three distinct real roots $\theta_0< \theta_1< \theta_2$. They satisfy
		\begin{align}
%			0 <\, & \alpha <  1 ,\\
			0 <\, & \theta_0 <  1 ,\\
%			a_1 - 1 <\, & \alpha_1  < a_1, \\
			a_1 - 1 <\, & \theta_1  < a_1, \\
%			a_2 <\, & \alpha_2 < a_2 + 1.
			a_2 <\, & \theta_2 < a_2 + 1.
		\end{align}
	\end{lem}
	\begin{proof}
		It is clear that $f(\mathbb{R}_{\leq 0}) \subseteq \mathbb{R}_{\leq 0}$.
		Observe that $f(0)=-1$ and $f(1) = (a_1-1)(a_2 - 1)-1 \geq 1$  since 
%		$a_1,a_2 \in \mathbb{Z}_{\geq 2}$ and $a_2 > a_1$. Hence $  0 < \alpha < 1  $.
		$a_1,a_2 \in \mathbb{Z}_{\geq 2}$ and $a_2 > a_1$. Hence $0 < \theta_0 < 1$.
		
		Next, since $f(a_1) = -1$ and $f(a_1-1) = (a_1-1)(a_2-a_1 + 1) - 1 \geq 1$, we have 
%		$a_1 -1 < \alpha_1 < a_1$.
		$a_1 -1 < \theta_1 < a_1$.
		Similarly, $f(a_2) = -1$ and $f(a_2+1) = (a_2 + 1)(a_2-a_1 +1) - 1 > 0$. Hence,
%		$a_2  < \alpha_2 < a_2+1$.
		$a_2  < \theta_2 < a_2+1$.
	\end{proof}
	
	{In the rest of the article,
%	for any polynomial $f$ satisfying \cref{eq:duke_poly}, the ordered embedding $\sigma : \mathbb{Q}[\alpha] \rightarrow \R^3$ we consider is given by $\sigma(\alpha) = (\alpha,\alpha_1,\alpha_2),$
	for any polynomial $f$ satisfying \cref{eq:duke_poly}, the ordered embedding $\sigma : \mathbb{Q}[\theta_0] \rightarrow \R^3$ we consider is given by $\sigma(\theta_0) = (\theta_0,\theta_1,\theta_2),$
%	where $\alpha<\alpha_1<\alpha_2$ are the three roots of $f$.}
	where $\theta_0< \theta_1< \theta_2$ are the three roots of $f$.}
	
%\noindent	Observe that the three integers $\alpha,\alpha-a_1, \alpha-a_2 \in \calO$ must 
\noindent Observe that the three elements $\theta_0,\theta_0-a_1, \theta_0-a_2 \in \calO$ must 
	be units 
	since by definition
	\begin{equation}
%		\alpha(\alpha-a_1)(\alpha-a_2) = 1.
		\theta_0(\theta_0-a_1)(\theta_0-a_2) = 1.
		\label{eq:unit-relation}
	\end{equation}
	In the upcoming lemmas, we will try to understand
%	the subgroup $\langle \alpha , \alpha-a_1 ,\alpha-a_2 \rangle \subseteq \calO^{\times}$.
	the subgroup $\langle \theta_0 , \theta_0-a_1 ,\theta_0-a_2 \rangle \subseteq \calO^{\times}$.

	\begin{lem}
		\label{le:regulator_of_units}
		Let 
		$y_1 = \log a_1$, $y_2 = \log a_2$ and $y_3 = \log ( a_2 - a_1 )$. Consider the map $\psi:\calO^{\times} \rightarrow \mathfrak{a}$ defined in  \cref{eq:defi_of_psi}. Then, we have that 
		\begin{equation}
%			\frac{1}{\sqrt{3}}\mathrm{covol}\Big(\psi\big( \langle \alpha, \alpha-a_1 ,\alpha-a_2\rangle \big)\Big) = 
			\frac{1}{\sqrt{3}}\mathrm{covol}\Big(\psi\big( \langle \theta_0, \theta_0-a_1 ,\theta_0-a_2\rangle \big)\Big) = 
			\det\left(  
			\begin{smallmatrix}
				y_1+y_2 & y_1 \\
				y_1 & y_1 +y_3
			\end{smallmatrix}  
			\right)
			+ \varepsilon,
		\end{equation}
		where $\varepsilon$ is an error term satisfying
		\begin{equation}
			| \varepsilon | \ll  \max\lbrace y_1,y_2,y_3\rbrace \cdot \max\Big\lbrace \tfrac{1}{a_1}, \tfrac{1}{a_2-a_1} \Big\rbrace,
			\label{eq:described_error}
		\end{equation}
		with the implicit constant independent of $a_1,a_2$.
	\end{lem}
	\begin{proof}
		Firstly, note that 
		\begin{equation}
%			\langle \alpha, \alpha - a_1, \alpha - a_2 \rangle 
			\langle \theta_0, \theta_0 - a_1, \theta_0 - a_2 \rangle 
			= 
%			\langle \alpha, \alpha - a_1 \rangle , 
			\langle \theta_0, \theta_0 - a_1 \rangle , 
		\end{equation}
		so we can ignore the third. 
%		The conjugates of $\alpha- a_i$ are $\alpha_1 - a_i$ and $\alpha_2-a_i$ where $i=1,2$ and $\alpha_1,\alpha_2$ are as in  \cref{le:control_the_roots}. 
		The conjugates of $\theta_0- a_i$ are $\theta_1 - a_i$ and $\theta_2-a_i$ where $i=1,2$ and $\theta_1,\theta_2$ are as in  \cref{le:control_the_roots}. 
%		The number $	\frac{1}{\sqrt{3}}\mathrm{covol}(\psi\left( \langle \alpha, \alpha-a_1\rangle \right))$ equals the absolute value of the determinant of any $2\times 2$ minor in
		The number $	\frac{1}{\sqrt{3}}\mathrm{covol}(\psi\left( \langle \theta_0, \theta_0-a_1\rangle \right))$ equals the absolute value of the determinant of any $2\times 2$ minor in
		\begin{equation}
			\left(
			\begin{matrix}
%				\log|\alpha | & \log |\alpha_1| & \log |\alpha_2| \\
				\log|\theta_0| & \log |\theta_1| & \log |\theta_2| \\
%				\log|\alpha  - a_1| & \log |\alpha_1 - a_1| & \log |\alpha_2 - a_1|\\
				\log|\theta_0 - a_1| & \log |\theta_1 - a_1| & \log |\theta_2 - a_1|\\
			\end{matrix}
			\right).
		\end{equation}
		Hence, we want to find the value of 
		\begin{equation}
			\det\left(
			\begin{matrix}
%				\log|\alpha | & \log |\alpha_1| \\
				\log|\theta_0| & \log |\theta_1| \\
%				\log|\alpha - a_1 | & \log |\alpha_1 - a_1| 
				\log|\theta_0 - a_1 | & \log |\theta_1 - a_1| 
			\end{matrix}
			\right).
		\end{equation}
%		Next, observe that $\log |\alpha | = - \log |\alpha - a_1| - \log |\alpha - a_2|$ and $\log|\alpha_1 - a_1| = - \log |\alpha_1| - \log |\alpha_1 - a_2|$ so we get that the determinant is equal to 
		Next, observe that $\log |\theta_0| = - \log |\theta_0 - a_1| - \log |\theta_0 - a_2|$ and $\log|\theta_1 - a_1| = - \log |\theta_1| - \log |\theta_1 - a_2|$ so we get that the determinant is equal to 
		\begin{equation}
			\det\left(
			\begin{matrix}
%				- \log|\alpha -a_1 | - \log|\alpha-a_2| & \log |\alpha_1| \\
				- \log|\theta_0 -a_1 | - \log|\theta_0-a_2| & \log |\theta_1| \\
%				\log|\alpha - a_1 | & - \log |\alpha_1| - \log |\alpha_1 - a_2|  \\
				\log|\theta_0 - a_1 | & - \log |\theta_1| - \log |\theta_1 - a_2|  \\
			\end{matrix}
			\right).
		\end{equation}
		
%		Now, we want to replace $\alpha_1$ with $a_1$. The difference in this approximation is 
		Now, we want to replace $\theta_1$ with $a_1$. The difference in this approximation is 
		\begin{equation}
%			| \log\alpha_1 - \log a_1 |  <  \log\left( \tfrac{a_1}{a_1 - 1} \right) \ll \tfrac{1}{a_1}.
			| \log\theta_1 - \log a_1 |  <  \log\left( \tfrac{a_1}{a_1 - 1} \right) \ll \tfrac{1}{a_1}.
		\end{equation}
		Similarly, one can show that for $i=1,2, j =1,2, i \neq j$
		\begin{align}
%			|\log | \alpha - a_i | - \log a_i| \ll & \tfrac{1}{a_1} ,\\
			|\log | \theta_0 - a_i | - \log a_i| \ll & \tfrac{1}{a_1} ,\\
%			|\log | \alpha_j - a_i | - \log |a_j- a_i|| \ll & \tfrac{1}{|a_i - a_j|}.
			|\log | \theta_j - a_i | - \log |a_j- a_i|| \ll & \tfrac{1}{a_2-a_1}.
		\end{align}
		For $i,j = 1,2$, we can then have some error terms $\varepsilon_{ij}$ with
		\begin{equation}
			|\varepsilon_{ij}| \ll  \max\{\tfrac{1}{a_1}, \tfrac{1}{a_2-a_1} \},
		\end{equation}
		such that the determinant can be written as
		\begin{equation}
			\det\left(
			\begin{smallmatrix}
				- y_1 - y_2 + \varepsilon_{11} & y_1 + \varepsilon_{12} \\
				y_1 + \varepsilon_{21} & - y_1 - y_3  + \varepsilon_{22} \\ 
			\end{smallmatrix}
			\right)
			=  
			\det\left(
			\begin{smallmatrix}
				- y_1 - y_2 & y_1  \\
				y_1   & - y_1 - y_3    \\
			\end{smallmatrix} 
			\right) 
			+ C \varepsilon_{0} \max\{\tfrac{1}{a_1}, \tfrac{1}{a_2-a_1}\}
			.
		\end{equation}
%		where $|\varepsilon_0 | \leq  \max\{y_1,y_2,y_3\}$ and $C>0$ is some constant large enough.
		where $|\varepsilon_0 | \leq  \max\{y_1,y_2,y_3\}$ and $C>0$ is an absolute constant.
	\end{proof}
	
	Observe that 
	\begin{equation}
		\det\left(
		\begin{smallmatrix}
			y_1 + y_2 & y_1 \\
			y_1 & y_1 + y_3
		\end{smallmatrix}
		\right)
		= y_1 y_2 + y_2 y_3 + y_1 y_3.
	\end{equation}
	
%	We will write another Lemma to calculate the discriminant of $\calO = \mathbb{Z}[\alpha]$.
	The following lemma calculates the discriminant of $\calO = \mathbb{Z}[\theta_0]$.
	\begin{lem}
		\label{le:discriminant}
		Let $y_1,y_2,y_3$ be as in \cref{le:regulator_of_units}. Then, we have that 
		\begin{equation}
			\frac{1}{2}\log (\disc\calO/4 )   = y_1 + y_2 + y_3 + \varepsilon,
		\end{equation}
		where 
		$
		| \varepsilon | \leq  C \max\left\{\tfrac{1}{a_1},\tfrac{1}{a_2-a_1}\right\}+\log 2$
		for some suitable $C > 0$.
	\end{lem}
	\begin{proof}
		As discussed in \cref{re:example_of_order}, we have
		\begin{equation}
%			\tfrac{1}{2}\log\disc(\calO) =    \log |\alpha_1- \alpha| + \log |\alpha_2- \alpha| + \log|\alpha_2 - \alpha_1|,
			\tfrac{1}{2}\log\disc(\calO) = \log |\theta_1- \theta_0| + \log |\theta_2- \theta_0| + \log|\theta_2 - \theta_1|,
		\end{equation}
%		with $\alpha,\alpha_1,\alpha_2$ as in \cref{le:control_the_roots}. Then, since we know from \cref{le:control_the_roots} that
		with $\theta_0,\theta_1,\theta_2$ as in \cref{le:control_the_roots}. Then, since we know from \cref{le:control_the_roots} that
		\begin{align}
%			a_1 -2	< \, & |\alpha_1 -\  \alpha|  < a_1, \\
			a_1 -2	< \, & |\theta_1 - \theta_0|  < a_1, \\
%			a_2 -a_1 < \, & |\alpha_2 - \alpha_1| < a_2 -a_1 + 2 ,\\
			a_2 -a_1 < \, & |\theta_2 - \theta_1| < a_2 -a_1 + 2 ,\\
%			a_2 -1 < \,& |\alpha_2 - \  \alpha |   < a_2 + 1.
			a_2 -1 < \,& |\theta_2 - \theta_0|   < a_2 + 1.
		\end{align}
%		We then again replace $\log(\alpha_1-\alpha) \simeq  y_1, \log(\alpha_2 - \alpha) \simeq y_2, \log(\alpha_2 - \alpha_1) \simeq y_3$. 
		We then again replace $\log(\theta_1-\theta_0) \simeq y_1$, $\log(\theta_2 - \theta_0) \simeq y_2$, and $\log(\theta_2 - \theta_1) \simeq y_3$. 
		If we keep track of the error, we get the required term.
	\end{proof}

	\begin{prop}\label{pr:proposition_polynom-units-orders}
		There exists a $C>0$ such that whenever $a_2-a_1, a_1 > C$, we get that the
%		group of units of the order $\mathcal{O} = \Z[\alpha]$ is generated by $\alpha$ and $\alpha-a_1$ i.e. 
		group of units of the order $\mathcal{O} = \Z[\theta_0]$ is generated by $\theta_0$ and $\theta_0-a_1$ i.e. 
%		$$\mathcal{O}^\times = \langle \alpha, \alpha - a_1,  \pm 1 \rangle.$$ 
		$$\mathcal{O}^\times = \langle \theta_0, \theta_0 - a_1,  \pm 1 \rangle.$$ 
	\end{prop}
	
	\begin{proof}
%		It is clear that $\alpha,\alpha - a_1,\alpha -a_2$ are non-torsional units in $\calO^{\times}$, since $\pm 1$ are the only torsional real units. Also, among $\alpha,\alpha-a_1,\alpha-a_2$, we get that any two of them generate the third since they multiply to be equal to 1. 
		It is clear that $\theta_0,\theta_0 - a_1,\theta_0 -a_2$ are non-torsional units in $\calO^{\times}$, since $\pm 1$ are the only torsional real units. Also, among $\theta_0,\theta_0-a_1,\theta_0-a_2$, any two generate the third by \cref{eq:unit-relation}. 
		
%		Let $\Theta \subseteq \calO^{\times}$ be the subgroup generated by $\alpha,\alpha-a_2$. 
		Let $\Theta \subseteq \calO^{\times}$ be the subgroup generated by $\theta_0,\theta_0-a_2$ and $\pm 1$. 
		Recall the map $\psi : \calO^{\times} \rightarrow \mathfrak{a}$ defined in \cref{ss:nt_prelims}.
		Let $R(\Theta)$ be the covolume of the lattice $\psi(\Theta) \subseteq \mathfrak{a}$ divided by $\sqrt{3}$.
		We know by \cref{th:cusick}
		that 
		\begin{equation}
			\frac{1}{16}   \log^2 (\disc\calO/4) 
			\leq R(\calO)   \leq R(\Theta)  .
		\end{equation}
		We can then substitute from \cref{le:discriminant} that
		\begin{align}
			\frac{4}{16} \log^2 (\disc\calO/4)  & =  (y_1+y_2 +y_3 + \varepsilon)^{2} \\
			& = y_1^{2} + y_2^{2} + y_{3}^{2} + 2(y_1y_2 + y_2y_3 + y_3 y_1) +2 \varepsilon(y_1 + y_2 + y_3)+\varepsilon^2\\
			& \geq 3(y_1 y_2 + y_2 y_3 + y_3 y_1) +2 \varepsilon(y_1 + y_2 + y_3),
		\end{align}
%		where the error term $\epsilon$ is bounded by $\max\{\tfrac{1}{a_1}, \tfrac{1}{a_2 - a_1}\}+\log 2$.
		where the error term $\varepsilon$ is bounded by $C\max\{\tfrac{1}{a_1}, \tfrac{1}{a_2 - a_1}\}+\log 2$.
		Note that in the last inequality, we are using
		\begin{equation}
			y_1^{2} + y_2^{2}+y_3^{2}  - y_1y_2 - y_2 y_3 - y_1y_3\geq \tfrac{1}{2}(y_1-y_2)^{2}+ \tfrac{1}{2}(y_1-y_3)^{2}+ \tfrac{1}{2}(y_2-y_3)^{2} \geq 0.
		\end{equation}
		On the other hand, we know by  \cref{le:regulator_of_units}
		that 
		\begin{equation}
			R(\Theta ) = y_1 y_2 + y_2 y_3 + y_3 y_1 + \varepsilon_1,
		\end{equation}
		where $\varepsilon _1 $ is as described in  \cref{eq:described_error}. Denote $\varepsilon_2 =\tfrac{1}{2} \varepsilon(y_1 + y_2 + y_3)$.

		Now, we know that 
		\begin{equation}
%			1 \leq   \frac{R(\Theta)}{R(\calO)} \leq \frac{  y_1y_2 + y_1 y_3 + y_2 y_3 + \varepsilon_1 }{ \tfrac{3}{4}( y_1 y_2 + y_1 y_3 + y_2 y_3  )+ \varepsilon_2} \leq  \tfrac{4}{3} \left(\frac{1+ \varepsilon_1{'}}{ 1 + \varepsilon_2'}\right).
			1 \leq   \frac{R(\Theta)}{R(\calO)} \leq \frac{  y_1y_2 + y_1 y_3 + y_2 y_3 + \varepsilon_1 }{ \tfrac{3}{4}( y_1 y_2 + y_1 y_3 + y_2 y_3  )+ \varepsilon_2} \leq  \tfrac{4}{3} \left(\frac{1+ \varepsilon_1{'}}{ 1 + \tfrac{4}{3}\varepsilon_2'}\right).
		\end{equation}
		Here, for $i=1,2$
		\begin{equation}
			|\varepsilon_i'| = \frac{|\varepsilon_i|}{y_1y_2 + y_2y_3 + y_3y_1} \ll \max\left\{\tfrac{1}{\log a_1}, \tfrac{1}{\log(a_2-a_1)} \right\}.
		\end{equation}
		The implicit constant is independent of $a_1$ and $a_2$.
		To show this last inequality, we use that $ y_1y_2+y_2y_3+y_3y_1\geq \max\{y_1,y_2,y_3\}\min\{y_1,y_2,y_3 \}$.
		
		Hence, it follows that if $a_1,a_2-a_1$ are large enough, then $R(\Theta)/R(\calO)$ must be an integer smaller than $5/3$. Hence, it is $1$ and $\Theta = \calO^{\times}$.
	\end{proof}

	\subsection{Producing a bounded set of shapes}
	
For all $a_1,a_2 \in \N$, denote by 
	\begin{equation}
		\label{eq:denote_order_a_1a_2}
		\mathcal{O}^{(a_1,a_2)} = {\mathbb{Z}[X]}/{\langle X(X-a_1)(X-a_2)-1 \rangle}.
	\end{equation}
	Let us try to locate the shape $[\Lambda(\mathcal{O}^{(a_1,a_2)})]$ of $\mathcal{O}^{(a_1,a_2)}$.
	
	\begin{lem}
		\label{le:where_is_shape}
		Let $a_1 , a_2 \in \N$ be such that $a_1,a_2-a_1 > C$ for the constant $C>0$ in  \cref{pr:proposition_polynom-units-orders}.
		Let $\lambda = \log a_1 / \log a_2$. Then the shape $[ \Lambda(\mathcal{O}^{(a_1,a_2)}) ]$ is the same as 
		\begin{equation}
%			\mathbb{Z}v_1 + \mathbb{Z} v_2  \ \subseteq \ \mathfrak{a} = \{x \in \mathbb{R}^{3} \mid x_1 + x_2 + x_3 = 0 )\},
			\mathbb{Z}v_1 + \mathbb{Z} v_2  \ \subseteq \ \mathfrak{a},
		\end{equation}
%		where the vectors $v_1,v_2$ are given by
		where $\mathfrak{a}$ is defined in \cref{eq:lie-algebra-A} and the vectors $v_1,v_2$ are given by
		\begin{align}
%			v_1 &=\psi(\alpha)/\log a_2   = (-1-\lambda,\lambda,1)  + \varepsilon_1, \\ 
			v_1 &=\psi(\theta_0)/\log a_2   = (-1-\lambda,\lambda,1)  + \varepsilon_1, \\ 
%			v_2 &=\psi(\alpha-a_1)/\log a_2  = (\lambda,-1-\lambda,1) + \varepsilon_2.
			v_2 &=\psi(\theta_0-a_1)/\log a_2  = (\lambda,-1-\lambda,1) + \varepsilon_2.
		\end{align}
		with $|\varepsilon_1|, |\varepsilon_2| \ll  \max\{\frac{1}{a_1}, \frac{1}{a_2},\frac{1}{a_2-a_1} + \frac{a_1}{a_2}\}$.
	\end{lem}
	\begin{proof}
		Recall the map $\psi$ defined in \cref{eq:defi_of_psi}.
%		According to \cref{pr:proposition_polynom-units-orders} because $C \gg 1$ is large enough, we can realize $\Lambda(\mathcal{O}^{(a_1,a_2)})$ as the $\mathbb{Z}$-span of $\psi(\alpha)$ and $\psi(\alpha-a_1)$. Writing diagonal matrices as $\mathbb{R}^{3}$, we are looking at the lattice spanned by 
		According to \cref{pr:proposition_polynom-units-orders} because $C \gg 1$ is large enough, we can realize $\Lambda(\mathcal{O}^{(a_1,a_2)})$ as the $\mathbb{Z}$-span of $\psi(\theta_0)$ and $\psi(\theta_0-a_1)$. Writing diagonal matrices as $\mathbb{R}^{3}$, we are looking at the lattice spanned by 
		\begin{align}
%			\psi(\alpha) & =  \left( -\log \alpha_2 -\log \alpha_1 \,,\, \log\alpha_1 \,,\, \log\alpha_2\right) ,\\
			\psi(\theta_0) & =  \left( -\log \theta_2 -\log \theta_1 \,,\, \log\theta_1 \,,\, \log\theta_2\right) ,\\
%			\psi(\alpha - a_1) & =  \left( \log|\alpha - a_1| \, , \, -\log|\alpha - a_1| -\log|\alpha_2 - a_1| , \log|\alpha_2 - a_1|  \right) .
			\psi(\theta_0 - a_1) & =  \left( \log|\theta_0 - a_1| \, , \, -\log|\theta_0 - a_1| -\log|\theta_2 - a_1| , \log|\theta_2 - a_1|  \right) .
		\end{align}
		In the following estimates, the $O$-notation refers to the regime $\min\{a_1,a_2-a_1\}\to\infty$.
%		We observe that $\log \alpha_1 = \lambda \log a_2 + O(\frac{1}{a_1})$, $\log \alpha_2 = \log a_2  + O(\frac{1}{a_2})$ and $\log|\alpha - a_1|=\lambda\log a_2+O(\frac{1}{a_1})  $, $\log|\alpha_2-a_1|= \log |a_1-a_2|  + O(\frac{1}{|a_2-a_1|}) = \log a_2 + O(\frac{1}{a_2-a_1} + \frac{a_1}{a_2})$. Since the shape point is invariant under scaling, we can divide these two vectors by $\log a_2$ and get the desired result.
		We observe that $\log \theta_1 = \lambda \log a_2 + O(\frac{1}{a_1})$, $\log \theta_2 = \log a_2  + O(\frac{1}{a_2})$ and $\log|\theta_0 - a_1|=\lambda\log a_2+O(\frac{1}{a_1})  $, $\log|\theta_2-a_1|= \log |a_1-a_2|  + O(\frac{1}{|a_2-a_1|}) = \log a_2 + O(\frac{1}{a_2-a_1} + \frac{a_1}{a_2})$. Since the shape point is invariant under scaling, we can divide these two vectors by $\log a_2$ and get the desired result.
	\end{proof}
	
%	Note that the angle between the vectors $v_1,v_2$ defined in \cref{le:where_is_shape} is roughly
	Note that the {cosine of the angle} between the vectors $v_1,v_2$ defined in \cref{le:where_is_shape} is {approximated by}
	\begin{equation}
		\label{eq:angle}
		\frac{\langle v_1 , v_2 \rangle}{\|v_1\| \|v_2\|} \simeq  -1+\frac{3}{2(1+\lambda + \lambda^{2})}.
	\end{equation}
%	As $\lambda$ varies between $0$ and $1$, the expression on the right varies between $-1/2$ and $1/2$. The error in this approximation is the same as the errors in \cref{le:where_is_shape}.
	As $\lambda$ varies between $0$ and $1$, the expression on the right varies between $-1/2$ and $1/2$. The error in this approximation admits a bound similar to those in \cref{le:where_is_shape}.
	
	This lets us claim the following proposition.
	First define
	\begin{equation}
		\mathcal{C} = \{ \calO^{(a_1,a_2)} \mid a_2 -a_1,a_1 > C\},
	\end{equation}
	where $C>0$ is the constant defined in \cref{pr:proposition_polynom-units-orders}.
	
	Let $\mathbb{H} \simeq \SL_{2}(\mathbb{R})/\SO_2(\mathbb{R})$ be the upper half space identified with $\{x + i y \mid y>0\} \subseteq \mathbb{C}$.
%	Denote $[\ ]_{\mathbb{H}}  : \SL_{2}(\mathbb{R})/\SL_{2}(\mathbb{Z}) \rightarrow \SL_{2}(\mathbb{Z})\backslash \mathbb{H}$ the canonical map given by $g \mapsto g^{-1} \cdot i$ where $\cdot $ is the action via fractional linear transformations and $i \in \mathbb{H}$ is the identity coset in the upper half plane. It is a classical fact that a fundamental domain in $\mathbb{H}$ for the action of $\SL_2(\mathbb{Z})$ is contained in
	Denote $[\ ]_{\mathbb{H}}  : \SL_{2}(\mathbb{R})/\SL_{2}(\mathbb{Z}) \rightarrow \SL_{2}(\mathbb{Z})\backslash \mathbb{H}$ the canonical map given by $g \mapsto g^{-1} \cdot i$ where $\cdot $ is the action via fractional linear transformations and $i \in \mathbb{H}$ is the identity coset in the upper half plane. The standard fundamental domain in $\mathbb{H}$ for the action of $\SL_2(\mathbb{Z})$ is
	\begin{equation}
%		\Big\lbrace x + i y \; \Big\vert \; |x| \leq \tfrac{1}{2} , |x^{2} + y^{2}| \geq 1 \Big\rbrace \subseteq \mathbb{H}.
		\Big\lbrace x + i y \; \Big\vert \; |x| \leq \tfrac{1}{2} , |x^{2} + y^{2}| \geq 1, y>0 \Big\rbrace \subseteq \mathbb{H}.
	\end{equation}
	We recover a result from David-Shapira \cite{davidShapesUnitLattices2017} and Cusick \cite{cusickRegulatorSpectrumTotally1991}.
	\begin{prop}\label{prop-bounded-shape}
		The closure of the set
		$ \lbrace [[\Lambda(\calO)]]_{\mathbb{H}} \rbrace_{\calO \in \calC} \subset \SL(2,\Z) \backslash \mathbb{H}$ contains 
		$\SL_{2}(\mathbb{Z}) B_{0}$ where 
		\begin{equation}
%			B_{0} = \Big\lbrace x + i y \; \Big\vert \; |x| \leq \tfrac{1}{2} , |x^{2} + y^{2}| = 1 \Big\rbrace.
			B_{0} = \Big\lbrace x + i y \; \Big\vert \; |x| \leq \tfrac{1}{2} , |x^{2} + y^{2}| = 1, y>0 \Big\rbrace.
		\end{equation}
	\end{prop}
	\begin{proof}
%		Following \cref{pr:proposition_polynom-units-orders}, given $a_1,a_2$ such that $a_2-a_1,a_1 > C$ we know where the shape lies due to \cref{le:where_is_shape}. For a fixed value of $\lambda$, we choose $a_1 = \lfloor a_2^{\lambda} \rfloor$ and let $a_2 \rightarrow \infty$.
		Following \cref{pr:proposition_polynom-units-orders}, given $a_1,a_2$ such that $a_2-a_1,a_1 > C$ we know where the shape lies due to \cref{le:where_is_shape}. For a fixed value of $\lambda \in (0,1)$, we choose $a_1 = \lfloor a_2^{\lambda} \rfloor$ and let $a_2 \rightarrow \infty$.
%		For this choice of $a_1,a_2$, the limiting shape $[ \Lambda(\mathcal{O}^{(a_1,a_2)}])$
		Along this sequence, the limiting shape $[ \Lambda(\mathcal{O}^{(a_1,a_2)}])$
		is just some lattice spanned by two vectors $v_1,v_2$ of length $\sqrt{2+ 2\lambda+2\lambda^{2}}$
		and angle between $\pi/3$ and $2\pi/3$ given by \cref{eq:angle}. As we vary $\lambda$ between 0 and 1, we get all the points in $\SL_2(\mathbb{Z}) B_0$ as claimed.
	\end{proof}

	\section{Units of suborders and a matrix congruence equation}
	\label{ss:suborders-matrix-eq}
	For any three distinct prime numbers $p_1, \; p_2, \; p_3$, any order $\calO \in \calC$, and all integer exponents $c,\; d, \; r \in \N$, denote by
	$$\calO_{c,d,r} = \Z + p_1^c p_2^d p_3^r\calO .$$
	Note that $\calO_{c,d,r}^\times = \calO^\times \cap \calO_{c,d,r}$.
	
	We first show in this lemma that the local behaviour of the three primes $p_1,p_2,p_3$ is independent of each 
	other.
	\begin{lem}
		\label{le:local_behaviour}
		We have 
		\begin{equation}
			\Z + p_1^c p_2^d p_3^r\calO = \left( \Z + p_1^{c}\calO\right) \cap \left( \Z + p_2^{d}\calO\right) \cap \left( \Z + p_3^{r}\calO\right)   .
		\end{equation}
	\end{lem}
	\begin{proof}
		It is clear that 
		\begin{equation}
			\Z + p_1^c p_2^d p_3^r\calO 
			\subseteq 
			\left( \Z + p_1^{c}\calO\right) \cap \left( \Z + p_2^{d}\calO\right) \cap \left( \Z + p_3^{r}\calO\right)   .
			\label{eq:intsection}
		\end{equation}
		Let us show the other direction. Suppose $b$ lies in the set on the right side of \cref{eq:intsection}. Then there exist $a_1,a_2,a_3 \in \mathbb{Z}$ such that
		\begin{align}
			b  \in (a_1 +  p_1^{c} \calO) \cap 
			(a_2 +  p_2^{d} \calO )\cap
			(a_3 +  p_3^{r} \calO).
		\end{align}
		But then, we can find a lift $a \in \mathbb{Z}$ {by the Chinese remainder theorem} such that 
		\begin{align}
			a \in (a_1 + p_1^{c} \Z )\cap (a_2 + p_2^{d} \Z )\cap (a_3 + p_3^{r} \Z ),
		\end{align}
		and this shows that $b \in a + p_1^{c}p_2^{d} p_3^{r} \calO$.
	\end{proof}
	
	We will now focus on what happens with respect to a single prime $p$. 
	\iffalse
	Here is a lemma about the multiplicative structure of the order $\calO$ modulo $p$.
	
	\begin{lem}
		\label{le:multiplicate_mod_p}
		Let $a_2 > a_1 > 0$ be some integers and $f(X) \in \mathbb{Z}[X]$ be the polynomial $X(X-a_1)(X-a_2)-1$. 
		Let $p$ be a prime such that $f(X)$ is separable\footnote{
			This condition is also equivalent to asking that $p$ be unramified with respect to $\calO$.
		} modulo $p$, that is every irreducible factor of $f(X)$ appears with multiplicity $1$.
		
		Then, for the order $\mathcal{O}= \mathbb{Z}[X]/\langle f(x) \rangle$ given in \cref{eq:denote_order_a_1a_2},
		the order of the group $(\mathcal{O}/p\mathcal{O})^{\times}$ is coprime to $p$.
	\end{lem}
	\begin{proof}
		We then have the isomorphisms
		\begin{equation}
			\frac{\mathcal{O}}{p\mathcal{O}} \simeq \frac{\mathbb{Z}[X]}{\langle p, f(X) \rangle } \simeq \frac{\mathbb{F}_p[X]}{ \langle f(X) \rangle}.
		\end{equation}
		Then, the structure of the ring $\calO/p\mathcal{O}$ will depend on how $f(X)$ factorizes modulo $p$. For example, if $f(X)$ splits into three linear factors modulo $p$, then $\calO/p\calO$ is isomorphic to three copies of $\mathbb{Z}/p\mathbb{Z} = \mathbb{F}_{p}$. 
		If $f(X)$ splits into a linear factor times a quadratic factor modulo $p$, then $\calO/p\calO$ is isomorphic to $\mathbb{F}_{p} \times \mathbb{F}_{p^{2}}$, where $\mathbb{F}_{p^{2}}$ is the unique finite field containing $p^{2}$ elements.
		
		In both cases, the order of the unit group $(\calO/p\calO)^{\times}$ is not divisible by $p$.
	\end{proof}
	\fi
	We will need the following lemma.
	\begin{lem}
		\label{le:order_mod_prime_power}
		Let $b \in \calO^{\times}$ be a non-torsional unit
		and $p$ be an odd prime.
		Then, there exist $l_1, l_2 \in \Z_{\geq 1}$ such that for all $c \geq 1$ and $n\in\Z$
		\begin{equation}
			b^{n} \in \Z + p^{c}\calO \Leftrightarrow n \in N\Z \text{ where }N =  l_1p^{\max ( c-l_2, 0)}. 
		\end{equation}

	\end{lem}
	\begin{proof}
		Since $b$ is a non-torsional unit, we know that for any $l\geq 1$, $b^{l} \notin \Z$. 
%		It is clear that by $b^{l} \in 1 + p \calO$ for some $l\geq 1$. Indeed, this is because $b$ maps to $(\calO/p\calO)^{\times}$. 
		There exists $l\geq 1$ such that $b^{l} \in 1 + p \calO$. Indeed, the image of $b$ in the finite group $(\calO/p\calO)^{\times}$ has finite order. 
		This $l$ is divisible by the following integer $l_1$:
		\begin{equation}
			\label{eq:l1def}
			l_1 = \min\{l \geq 1\  \mid\  b^{l}\in \Z + p^{l'} \calO \text{ for some }l' \geq 1\}.
		\end{equation}
		Then let $l_2$ be
		\begin{equation}
			\label{eq:l2def}
			l_2 = \max \{ l \geq 1 \mid  b^{l_1} \in \Z+p^{l}\calO\}.
		\end{equation}
		The statement is then clear for $c \leq l_2$. Now we will show that the statement is valid for $c+1$ by assuming it true for $c$. 
		Let $b^{n}   \in \Z + p^{c+1} \calO$.
		This implies that 
		$b^{n} \in \Z + p^{c+1} \calO \subseteq \Z + p^{c} \calO$.
		We know by our induction hypothesis that $n =  N m$ for some $m \in \mathbb{Z}$ where $N= l_1p^{c-l_2}$. 
		
		Let us assume that for some $a \in \Z \setminus p \Z$ and some $y \in \calO \setminus \Z$, we have
		$b^{N}= a + p^{c}y,$ for $p \nmid a$ and $y \notin \mathbb{Z} + p \mathcal{O}$.
		This implies that 
		\begin{equation}
			b^{Nm} \in a^{m} + m a^{m-1}p^{c}y + p^{2c} \calO \subseteq ma^{m-1} p^{c} y +  \left(  \Z + p^{2c}\calO \right). 
		\end{equation}
		Since $y \notin \Z$ and $p \nmid a$, the only way $b^{Nm} \in \Z + p^{c+1} \calO$ is if $p \mid m$. Hence, we prove that $Np \mid n$ and we are done.
		Conversely, if $p\mid m$, then both $ma^{m-1}p^cy$ and $p^{2c}\calO$ lie in $p^{c+1}\calO$, so
  $b^{Nm}\in\Z+p^{c+1}\calO$.
	\end{proof}

	We will deal with the case $p=2$ separately below.
\begin{lem}
\label{le:order_mod_prime_powertwo}
For $p=2$, take $l_{1},l_{2}$ as defined in (\ref{eq:l1def}) and (\ref{eq:l2def}) respectively.
Let $b \in \mathcal{O}^{\times}$.
Suppose that either $l_2\geq 2$,
or $l_2=1$ and
\[
b^{2l_1}\in\mathbb Z+4\calO
\qquad\text{but}\qquad
b^{2l_1}\notin\mathbb Z+8\calO.
\]
Then, the same conclusion as Lemma \ref{le:order_mod_prime_power} holds. That is, for every $c\geq1$ and $n\in\mathbb Z$,
  \[
  b^n\in\mathbb Z+2^c\calO
  \quad\Longleftrightarrow\quad
  n\in l_1\,2^{\max(c-l_2,0)}\mathbb Z.
  \]

\end{lem}
  \begin{proof}
	    The proof of \cref{le:order_mod_prime_power} works unchanged for $p=2$
  once $c\geq2$. If $l_2\geq2$, we therefore apply that proof starting
  with $c=l_2$.

  Suppose now that $l_2=1$. The calculation in the previous proof shows
  that
  \[
  b^n\in\Z+4\calO \quad\Longleftrightarrow\quad 2l_1\mid n.
  \]
  The additional assumption says that $b^{2l_1}\notin\Z+8\calO$.
  Thus we may apply the same proof starting with $c=2$ and $N=2l_1$.
  The result follows by induction.
  \end{proof}

	\begin{prop}\label{prop:suborder}
%	{\nihar This proposition needs to be a bit clearer about $M$. We are not sensitive about $\pm 1$.}
		Let $p_{1},p_{2},p_{3}$ be three distinct primes.
		If one of $p_{1},p_{2}$ or $p_{3}$ is not an odd prime, we additionally assume that the corresponding unit in the proof satisfies the additional hypotheses in~\cref{le:order_mod_prime_powertwo}.	
		
		Then there exists $(l_i,j_i)_{1 \leq i \leq 3}$ such that we have the following. 
		Fix $c,d,r\in \N$ and any order $\calO\in \calC$ such that the coefficients $a_1,a_2$ of its polynomial (as in \cref{pr:proposition_polynom-units-orders}) satisfy
		the following system,
		$$\left\lbrace \begin{array}{ cccccl}
			a_1 &\equiv&0&\equiv& a_2-1 &\pmod{p_1^c}\\
			a_1-1 &\equiv&0&\equiv& a_2 &\pmod{ p_2^d}\\
			a_1-1 &\equiv&0&\equiv& a_2-1 &\pmod{ p_3^r}
		\end{array} \right.$$
		
		\noindent Then the units of the suborder $\calO_{c,d,r}^\times$ are given by %$L_{c,d,r}\subset \Z^2$ in the sense that 
		$$\calO_{c,d,r}^\times =\left\lbrace \pm b_1^m b_2^n \; \bigg\vert  (m,n)  \in L_{c,d,r} \right\rbrace$$
%		where $b_1 = \alpha$, $b_2 = \alpha-a_1$ and $L_{c,d,r}$ denotes the solutions to the congruence equations {
		where $b_1 = \theta_0$, $b_2 = \theta_0-a_1$ and $L_{c,d,r}$ denotes the solutions to the congruence equations {
			$$\begin{array}{ cccl}
				m+n &\equiv &0 &\pmod{ l_1 p_1^{\max( c-j_1,0)}},\\
				m-2n &\equiv &0 &\pmod{ l_2 p_2^{\max( d-j_2,0)}},\\
				n-2m &\equiv &0 &\pmod{ l_3 p_3^{\max( r-j_3,0)}}.
			\end{array}  $$}
	\end{prop}
	\begin{proof}
%		Denote $b_3 = \alpha - a_2$. We have $b_1b_2b_3 = 1$ by construction.  
		Denote $b_3 = \theta_0 - a_2$. We have $b_1b_2b_3 = 1$ by construction.  
		From \cref{pr:proposition_polynom-units-orders}, we know that $$\calO^\times = \{\pm~b_1^m b_2^n \mid (m,n) \in \Z^2\}.$$ 
		We want to calculate $\calO_{c,d,r}^{\times}\subseteq\calO^{\times}$. We have
\begin{align}
		\calO_{c,d,r}^{\times} 
		 & = \{\pm~b_1^m b_2^n \mid \pm ~b_1^m b_2^n \in \Z + p_1^c p_2^d p_3^r \calO \} \\
		 & = \{\pm 1\} \cdot \{b_1^m b_2^n \mid b_1^m b_2^n \in \Z + p_1^c p_2^d p_3^r \calO \}.
		\end{align}
		
		We know from \cref{le:local_behaviour}
		that 
		\begin{align}
			& \ \ b_1^{m}b_2^{n}\in \Z + p_1^{c}p_2^{d}p_3^r\calO\\
			\Leftrightarrow 
			& \ \ b_1^{m}b_2^{n}\in 
			( \Z + p_1^{c}\calO
			)\cap( \Z + p_2^{d}\calO
			)\cap( \Z + p_3^r\calO
			)	.
		\end{align}
		
		Observe that the local conditions on $a_1,a_2$ imply the following:
		\begin{align}
			b_1\equiv b_2 \pmod{p_1^{c}\calO}\\
			b_1\equiv b_3 \pmod{p_2^{d}\calO}\\
			b_2\equiv b_3 \pmod{p_3^{r}\calO}.
		\end{align}
		Writing $b_3=b_1^{-1}b_2^{-1}$ and using \cref{le:order_mod_prime_power} and \cref{le:order_mod_prime_powertwo}, we get that 
		\begin{align}
			b_1^{m}b_2^{n} \in \Z + p_1^{c}\calO &\Leftrightarrow l_1 p_1^{\max( c-j_1,0 )} \mid (m+n)\\
			b_1^{m}b_2^{n} \in \Z + p_2^{d}\calO &\Leftrightarrow 
			l_2 p_2^{\max( d-j_2,0 )} \mid (m-2n)\\
			b_1^{m}b_2^{n} \in \Z + p_3^{r}\calO &\Leftrightarrow 
			l_3 p_3^{\max( r-j_3,0 )} \mid (n-2m).
		\end{align}
		
		{For $p_2$, due to $b_2\equiv(b_3b_1)^{-1}\equiv b_1^{-2}\pmod{ p_2^d\calO}$, we obtain $b_1^mb_2^n\equiv b_1^{m-2n}\pmod{p_2^d\calO}$. Then the condition we need is $b_1^{m-2n}\in \Z+p_2^d\calO $. Combined with \cref{le:local_behaviour}, we obtain the second congruence condition.
			Similarly for $p_3$, we get $b_1^{m}b_{2}^{n} \equiv b_{2}^{n-2m} \pmod{p_3^r\calO}$.
		}
	\end{proof}
	
	\begin{rem}
	From the geometric perspective, taking suborders is equivalent to taking Hecke neighbors of a periodic torus. The shapes of new tori are given by the shapes of units of suborders.
	\end{rem}

	\section{Density of a family of shapes}
	\label{ss:density}
In this section we set the triplet of primes $(p_1,p_2,p_3)=(2,3,5)$ in \cref{prop:suborder}.
We first calculate the sub-lattice $L_{c,d,r}$ for all $c,d,r \geq 2$.
In \cref{defin-Lcde}, we express $L_{c,d,r}$ as the image of a sub-lattice of $\Z^2$ under the action of some large diagonal matrix.
This expression motivates us to study, up to a fixed matrix action, the family of shapes of $\lbrace L_{c,d,r} \rbrace_{c,d,r \in \N}$ and prove their density in \cref{theorm_density}.
Finally, in the last section, we prove \cref{co:ds}. 
Namely, we prove that the family of shapes given by units of orders of the form $\calO_{c,d,r}=\Z + 2^c 3^d 5^r \calO$ where $\calO \in \lbrace \calO^{(a_1,a_2)}\rbrace_{a_1,a_2 \in \N}$ (Cf. \eqref{eq:denote_order_a_1a_2}) is dense.

	\subsection{The suborders matrix equations for \texorpdfstring{$(2,3,5)$}{(2,3,5)}}
	In this section we set the triplet of primes $(p_1,p_2,p_3)=(2,3,5)$ in \cref{prop:suborder} and fix some values of $c,d,r \geq 2$.
	Let us try to evaluate the values of $l_1,l_2,l_3$ and $j_1,j_2,j_3$ using the proof of \cref{le:order_mod_prime_power}.
	
	When $p_1=2$, we get from the local conditions imposed in \cref{prop:suborder} that
	$$f(X) = X(X-a_1)(X-a_2) - 1 \equiv X^3-X^{2}-1 \pmod{2}.$$ This is an irreducible polynomial. We can then calculate that in $\mathbb{F}_2[X]/\langle X^{3}+X^{2}+ 1 \rangle$ the multiplicative order of $X$ is 7. 
  In $\Z[X]/\langle X^3-X^2-1\rangle$, one computes
  \[
  X^7=3+2X+4X^2 \text{ and }
  X^{14}=41+28X+60X^2.
  \]
  Thus $X^7$ is in $\mathbb{Z}$ modulo $2$ but not modulo $4$, so
  $l_1=7$ and $j_1=1$. Moreover, $X^{14}$ is in  $\mathbb{Z}$ modulo $4$
  but not modulo $8$. Hence the additional hypothesis of the preceding
  lemma \cref{le:order_mod_prime_powertwo} for $p=2$ is satisfied.

	Similarly, for $p_2=3$ we conclude that 
	$$f(X)=X(X-a_1)(X-a_2) - 1 \equiv X^3-X^{2}-1 \pmod{3}.$$
	The smallest degree polynomial of the type $X^{n_1}-n_2$ for $n_1\in \mathbb{Z}_{\geq 1}, n_2 \in \mathbb{Z}$ that is divisible by $X^{3}-X^{2} -1$ modulo $3$ is $X^{8} - 1$. This means that $l_2=8$ and one can check that $j_2=1$ since $X^8-1$ is divisible by $3$ but not by $9$.
	
	Finally, for $p_3=5$ we get with the values of $a_1,a_2$ in \cref{prop:suborder} that
	$$X(X-a_1)(X-a_2) - 1 \equiv X^3-2X^{2} + X-1 \pmod{5}.$$ Then, we observe that the smallest power of $(X-1)$ that is an integer modulo $X^{3}-2X^{2}+X-1$ is $(X-1)^{24}$ which implies that 
	$l_3=24$. We can also check that $(X-1)^{24}-1$ is only once divisible by $5$ so $j_3=1$ again.

%	So we managed to compute $(l_1,l_2,l_3) =(7,8,24)$ and $(j_1,j_2,j_3)=(1,1,1)$. All our computations were carried out in \texttt{sagemath}. Based on these calculations and \cref{prop:suborder}, let us denote the set
	So we managed to compute $(l_1,l_2,l_3) =(7,8,24)$ and $(j_1,j_2,j_3)=(1,1,1)$. All our computations were carried out in \texttt{SageMath} and are formalized in the \texttt{Lean} verification.
	Let us denote the set
	\begin{equation}\label{eq:defi_of_L}
		L_{c,d,r}=
		\left\lbrace\begin{tabular}{c|ccl}
\multirow{3}*{$\begin{pmatrix}
    m \\
    n
\end{pmatrix} \in \Z^2$} 
& $m+n$  &$\equiv\,0$ &  $\pmod {7\cdot 2^{c-1}}$ \\ 
& $m-2n$ &$\equiv\,0$ &  $\pmod {8\cdot 3^{d-1}}$ \\ 
& $2m-n$ &$\equiv\,0$ &  $\pmod {24\cdot 5^{r-1}}$ 
		\end{tabular}\right\rbrace
	\end{equation}

	\begin{lem}\label{defin-Lcde}
		Denote by 
		$$g = 8 
		\left( 
		\begin{array}{rr}
			1 & 1\\
			0 & -1  
		\end{array}\right)
		\begin{pmatrix}
			3 & 0 \\
			0 & 1
		\end{pmatrix}
		\left( \begin{array}{rr}
			-1 & 1 \\
			2 & -1
		\end{array}\right).$$
		Then for every $c\geq 4$ and $d\geq 2$ and $r\geq 1$, the sub-lattice $L_{c,d,r}$ as defined in \cref{eq:defi_of_L} satisfies 
		$$L_{c,d,r} = g
		\begin{pmatrix}
			3^{d-2} & 0 \\
			0 & 5^{r-1}
		\end{pmatrix} 
		\bigg\lbrace 
		\binom{x}{y} \in \Z^2
		\; \bigg\vert \;
		3^{d-2} x - 5^{r-1} y \equiv 0 \pmod{ 7 \cdot 2^{c-4}}
		\bigg\rbrace.$$
	\end{lem}
	\begin{proof}
		By \cref{eq:defi_of_L}, $L_{c,d,r}$ is the set of integer vectors $(m,n) \in \Z^2$ that solve a rank-two equation in $\Z^3$, i.e., such that 
		$$ 
		\left(\begin{array}{rr}
			1 & 1 \\
			1 & -2 \\
			2 & -1
		\end{array}\right)
		\binom{m}{n} \in 
		\left(\begin{array}{ccc}
			7 \cdot 2^{c-1} & & (0) \\
			& 8 \cdot 3^{d-1} & \\
			(0) & & 24 \cdot 5^{r-1}
		\end{array}\right) \Z^3 .$$
		We compute the Smith normal form on the $3 \times 2$ matrix on the left and denote by 
		$$ 
		U = \left( 
		\begin{array}{rrr}
			0 &  -1 & 1\\
			0 & 2 &  -1\\
			1 & 1 &  -1
		\end{array}\right) , \quad
		M = 
		\left(
		\begin{array}{rr}
			1 & 1 \\
			1 & -2 \\
			2 & -1 
		\end{array}\right), \quad
		V = \left( 
		\begin{array}{rr}
			1 & 1\\
			0 & -1  
		\end{array}\right).
		$$ 
		Then $U \in \GL(3,\Z)$ and $V \in \GL(2,\Z)$ and
		$     M = U^{-1} \left(
		\begin{array}{rr}
			1 & 0 \\
			0 & 3 \\
			0 & 0
		\end{array} \right) V^{-1}.
		$
		Now, $V^{-1} L_{c,d,r}$ is the set of $(m,n)\in \Z^2$ such that there exists $(x,y,z) \in \Z^3$ such that
		\begin{equation}\label{eq:nicer-Lcde}
			\left(
			\begin{array}{rr}
				1 & 0 \\
				0 & 3 \\
				0 & 0
			\end{array} \right) 
			\binom{m}{n} =
			U  \left(\begin{array}{r}
				7 \cdot 2^{c-1} \; x  \\
				8 \cdot 3^{d-1} \; y \\
				24 \cdot 5^{r-1}\; z
			\end{array}\right)  . 
		\end{equation}
		The first two lines of the matrix equation \cref{eq:nicer-Lcde} yield the following rank-$2$ linear equation
		\begin{equation}\label{eq:Lcde-2square}
			\left( 
			\begin{array}{rr}
				1 & 0\\
				0 & 3
			\end{array}\right)
			\binom{m}{n} =
			\left( 
			\begin{array}{rr}
				-1 & 1\\
				2 &  -1
			\end{array}\right)
			\binom{8 \cdot 3^{d-1} \; y}{24 \cdot 5^{r-1}\; z}. 
		\end{equation}
		Hence using that $d \geq 2$, we deduce that 
		$V^{-1} L_{c,d,r} \subset 
		8 \begin{pmatrix}
			3 & 0 \\
			0 & 1
		\end{pmatrix} 
		\left( 
		\begin{array}{rr}
			-1 & 1\\
			2 &  -1
		\end{array}\right)
		\begin{pmatrix}
			3^{d-2} & 0\\
			0 &  5^{r-1}
		\end{pmatrix} \Z^2.
		$ 
		
		The last line of the matrix equation \cref{eq:nicer-Lcde} is a hyperplane equation in $\Z^3$. 
		It will induce a sub-lattice in $\Z^2$ which will be mapped to $L_{c,d,r}$ by an element in $\GL(2,\R)$.
		By the matrix equation, $(x,y,z)\in \Z^3$ satisfy
		\begin{equation}\label{eq:Lcde-hyperplane}
			7 \cdot 2^{c-1} \; x + 8 \cdot 3^{d-1} \; y - 24 \cdot 5^{r-1}\; z = 0.
		\end{equation}
		Equivalently,
		$$ -7 \cdot 2^{c-4} \; x = 3^{d-1} \; y - 3 \cdot 5^{r-1}\; z. $$
		Since $c \geq 4$ and using that $3$ and $7 \cdot 2^{c-4}$ are coprime and $d \geq 2$, the integers $(y,z)\in \Z^2$ solve the congruence equation
		$ 3^{d-2} \; y -  5^{r-1}\; z \equiv 0 \; \pmod{ 7 \cdot 2^{c-4}}.$
		
		Finally, combining the hyperplane relation with the rank $2$ part, we deduce that
		$$ L_{c,d,r} = 8 V \begin{pmatrix}
			3 & 0 \\
			0 & 1
		\end{pmatrix} 
		\left( 
		\begin{array}{rr}
			-1 & 1\\
			2 &  -1
		\end{array}\right)
		\begin{pmatrix}
			3^{d-2} & 0\\
			0 &  5^{r-1}
		\end{pmatrix} \bigg\lbrace 
		\binom{y}{z}
		\; \bigg\vert \;
		3^{d-2} y - 5^{r-1} z \equiv 0 \pmod{ 7 \cdot 2^{c-4}}
		\bigg\rbrace. $$
	\end{proof}
	
\subsection{Density of shapes for the solutions of the \texorpdfstring{$(2,3,5)$}{(2,3,5)}-equations}

	For all $c,d,r\in \N$ with $c\geq 4$, $d\geq 2$ and $r \geq 1$, denote by 
	$$\Lambda_{c,d,r} = \begin{pmatrix}
		3^{d-2} & 0\\
		0 &  5^{r-1}
	\end{pmatrix} \bigg\lbrace 
	\binom{y}{z}
	\; \bigg\vert \;
	3^{d-2} y - 5^{r-1} z \equiv 0 \pmod{ 7 \cdot 2^{c-4}}
	\bigg\rbrace.$$
	Note that
	$L_{c,d,r} = g \Lambda_{c,d,r}$, where $g$ is defined in \cref{defin-Lcde}.
 The sets $L_{c,d,r}$ and $\Lambda_{c,d,r}$ are auxiliary lattices in $\R^2$ that, up to a linear transformation, give the sets of periods of periodic tori, as explained in the proof of \cref{co:ds} in \cref{ss:proof_of_density}. 
 For now, we prove that these lattices are dense.
	\begin{prop}\label{theorm_density}
		The collection of points $  \lbrace   [\Lambda_{c,d,r}] \rbrace_{c,d,r\in \N} \subset \SL(2,\R)/\SL(2,\Z)$ is dense.
	\end{prop}
	\begin{proof}
		
		Let $\Omega\subseteq \SL_2(\mathbb{R})/\SL_2(\Z) $ be the closure of $\{[\Lambda_{c,d,r}],\ c,d,r\in\N \}$. 
		We will first prove that $\Omega$ is invariant under the diagonal flow $\{a(t)\}_{t\in \R}$,
		then we prove it contains the horospherical piece $u([0,1])\SL_2(\Z)$, where 
		$$a(t)=\begin{pmatrix}
			e^{t/2} & 0 \\ 0 & e^{-t/2}
		\end{pmatrix}
		\,,\,\,u(s)=\begin{pmatrix}
			1 & s \\ 0 & 1
		\end{pmatrix}.$$
		Then we can apply Margulis' ``banana trick''
		to conclude the density of $\Omega$. (See for example \cite[Proposition 2.2.1 and Proposition 2.4.8]{kleinbock1996bounded} or \cite[Theorem 1.4 ]{shah_1996}).

\textbf{Diagonal invariance:}
Fix $t \in\mathbb{R}$. 
 We want to show that $\Omega$ is invariant under $a(t)$. It is sufficient to show that for every $[\Lambda_{c,d,r}]$ where $c,d,r\in \N$, we have $a(t) [\Lambda_{c,d,r}] \in \Omega$. By a diagonalization argument, it suffices to prove that there exist $(c_n,d_n,r_n)$ such that $[\Lambda_{c_n,d_n,r_n}]\rightarrow a(t)[\Lambda_{c,d,r}]$. 
Note that 
	$$a(t)\Lambda_{c,d,r} = \begin{pmatrix}
		3^{d-2}e^{t/2} & 0\\
		0 &  5^{r-1}e^{-t/2}
	\end{pmatrix} \bigg\lbrace 
	\binom{y}{z}
	\; \bigg\vert \;
	3^{d-2} y - 5^{r-1} z \equiv 0 \pmod{ 7 \cdot 2^{c-4}}
	\bigg\rbrace.$$
We will choose $c_n,d_n,r_n$ such that the rank-$2$ lattice solving the hyperplane equation $3^{d_n-2} y - 5^{r_n-1} z \equiv 0 \pmod{ 7 \cdot 2^{c_n-4}}$
remains satisfied and such that we also have the following convergence in $\mathrm{PGL}(2,\R)$: 
% i.e. by ensuring the second coefficient on the diagonal equals $1$ 
\begin{eqnarray}\label{eq-diagonal-convergence}
\begin{pmatrix}
 3^{d_n -2} & 0 \\
0 & 5^{r_n-1}
\end{pmatrix} \sim \begin{pmatrix}
 3^{d_n -2}/5^{r_n-1} & 0 \\
0 & 1 
\end{pmatrix}
&\xrightarrow[\; n \rightarrow +\infty \;]{} 
\begin{pmatrix}
e^{t}3^{d -2}/5^{r-1} & 0 \\
0 & 1 
\end{pmatrix} \sim
\begin{pmatrix}
3^{d -2}e^{t/2} & 0 \\
0 & 5^{r - 1}e^{-t/2}
\end{pmatrix}.
\end{eqnarray}

Let $c_n = c$ for all $n\geq 1$. 
Denote by $s_1$ (resp. $s_2$) the order of $3$ (resp. $5$) in $(\Z/7 \cdot 2^{c-4}\Z)^\times$.
To preserve the hyperplane equation, $(d_n,r_n)$ must satisfy the following conditions. 
	$$\left\lbrace
	\begin{array}{cr}
	3^{d_n} \equiv 3^d &\pmod{ 7 \cdot 2^{c-4}}\\
	5^{r_n} \equiv 5^{r} &\pmod{ 7 \cdot 2^{c-4} }
	\end{array}
	\right. \Longleftrightarrow 
	\left\lbrace
	\begin{array}{cr}
	d_n \equiv d &\pmod{s_1}\\
	r_n \equiv r &\pmod{s_2}
	\end{array}
	\right. .$$		
Denote by $p_n = (d_n -d)/s_1$ and $q_n = (r_n-r)/s_2$.
Taking logarithms in \eqref{eq-diagonal-convergence}, we deduce that the positive integers $(p_n,q_n)$ must equivalently satisfy the following convergence.
\begin{eqnarray}\label{eq-rational-approx}
p_n \cdot s_1\log(3) - q_n \cdot s_2 \log(5) - t \xrightarrow[\; n \rightarrow +\infty \;]{}  0.
\end{eqnarray}

Since $s_1$ and $s_2$ are positive integers, the ratio of $s_1 \log(3)$ and $s_2 \log(5)$ is irrational.
Hence every forward orbit of the translation $s_1 \log(3)$ in $\R/( s_2 \log(5)\Z) $ is dense. In particular, as $p_n,q_n$ vary in $\Z_{
\geq 0 }$, the left-hand side in \eqref{eq-rational-approx} starting from $-t$ visits infinitely often any neighbourhood of $0$.
Therefore, we may pick a sequence of positive integers $(p_n,q_n)$ with $p_n, q_n \rightarrow + \infty$, satisfying the convergence \eqref{eq-rational-approx}.
Then $d_n = d + s_1 p_n$ and $r_n = r+s_2  q_n$ are positive integers and we have constructed a sequence $\big([\Lambda_{c,d_n,r_n}]\big)_{n\geq 1}$ that converges towards $a(t)[\Lambda_{c,d,r}]$.

\textbf{Horospherical piece:}
For every integer $c\geq 3$, denote by $S_c$ the finite set of representatives in $[0,7\cdot 2^c]\cap \N$ of the multiplicative group generated by $5$ in $\big(\Z /(7 \cdot 2^c)\Z\big)^\times$, namely,	
$$S_c = \Big\lbrace f \in [0,7 \cdot 2^c]\cap \N \; \Big\vert \; \exists r \in \N \; \text{such that} \; 5^r \equiv f \pmod{7 \cdot 2^c} \Big\rbrace. $$ 		
The first step is to prove that as $c \rightarrow +\infty$ the set $\frac{S_c}{7 \cdot 2^c}$ becomes dense in $[0,1]$.

Then we prove that for every $c\in\N$ with $c\geq 3$,
\begin{eqnarray}\label{eq-horos-inclusion}
\begin{pmatrix}
1 & \frac{S_c}{7 \cdot 2^c} \\
0 & 1
\end{pmatrix} \Z^2 \subset \Omega.
\end{eqnarray}
 
By combining the two steps in a diagonal argument, we can deduce that $\Omega$ contains the horospherical piece $u([0,1])\Z^2$.

Let us first calculate the cardinality of $S_c$, i.e., the order of $5$ in the multiplicative group $\big(\Z / 7 \cdot 2^c\Z\big)^\times$.
By induction, starting from the fact that $5^{2^1}-1 = 2^3 \cdot 3$ and using that for every $k\geq 1$ 
$$5^{(2^k)} - 1 = \Big( 5^{(2^{k-1})}-1 \Big) \Big(5^{(2^{k-1})}+1 \Big),$$ 
we deduce that for every $c\geq 3,$ the order of $5$ in $(\Z/2^c\Z)^\times$ is $2^{c-2}$.
Furthermore, $5$ is a generator of $(\Z/7\Z)^\times$ and has order $2 \cdot 3$.
Therefore, $S_c$ is a set of $3 \cdot 2^{c-2}$ elements.
When $r$ is even (resp. odd), then $5^r\equiv 1 \pmod{2^3}$ and $5^r \equiv 1,2,4 \pmod 7 $ (resp. $\equiv 5 \pmod{2^3}$ and $\equiv  3,5,6  \pmod 7 $ ). 
By definition, any element in $S_c$ will satisfy either the even system of congruence equations, or the odd one. 
By the Chinese remainder theorem, there are $3 \cdot 2^{c-3}$ integer solutions to the even system of congruences in $(0, 7 \cdot 2^c)$, and likewise for the odd one. 
Consequently, the two systems of congruences describe all the elements in $S_c$.
Indeed, by explicit computation, for every integer $c\geq 3$
$$S_c = \big\lbrace  7 \cdot 2^3 k + j \; \big\vert \; j\in \lbrace 1,5,9,13,25,45 \rbrace \text{ and } k\in \Z_+ \cap [0,2^{c-3}-1]  \big\rbrace.$$
Hence $\min S_c = 1$ and $\max S_c = 7 \cdot 2^c - 11$ and the gaps between consecutive elements are at most $20$.
Consequently, $S_c/(7 \cdot 2^c)$ is dense in $[0,1]$ as $c$ goes to infinity.

Now, for the second step of the proof, we fix an integer $c\geq 3$ and an element $f \in S_c$.
Let us denote by $s_1$ (resp. $s_2$) the order of $3$ (resp. $5$) in the multiplicative group $\big(\Z /7\cdot 2^c\Z\big)^\times$.
We want to find a sequence of positive integers $(d_n,r_n)$ such that
$\big([\Lambda_{c+4,d_n+2,r_n+1}]\big)_{n\geq 1}$ converges towards $u\big( \frac{f}{7 \cdot 2^c} \big)\Z^2$. 
Denote by $r_f$ an integer such that $5^{r_f} \equiv f \pmod{7 \cdot 2^c}$. 
Assume $d_n,r_n$ satisfy the following equivalent conditions.
	$$\left\lbrace
	\begin{array}{cr}
	3^{d_n} \equiv 1 &\pmod{ 7 \cdot 2^{c}}\\
	5^{r_n} \equiv f &\pmod{ 7 \cdot 2^{c} }
	\end{array}
	\right. \Longleftrightarrow 
	\left\lbrace
	\begin{array}{lr}
	d_n \equiv 0 &\pmod{s_1}\\
	r_n \equiv r_f &\pmod{s_2}
	\end{array}
	\right. .$$
Then $ 3^{d_n} y - 5^{r_n}z \equiv y - f z \pmod{7 \cdot 2^c}$ for all $(y,z) \in \Z^2$. 
Hence the solutions to the hyperplane equation for $\Lambda_{c+4,d_n+2,r_n+1}$, noting that $7 \cdot 2^c \vert y - fz$ if and only if $y \in 7 \cdot 2^c \Z + fz $ can be written as follows.  
$$\bigg\lbrace 
	\binom{y}{z}
	\; \bigg\vert \;
	3^{d_n} y - 5^{r_n} z \equiv y - f z \equiv 0 \pmod{ 7 \cdot 2^{c}}
	\bigg\rbrace = 
\begin{pmatrix}
7 \cdot 2^c & f \\
0 & 1	
\end{pmatrix}	\Z^2. $$	
As a consequence, 
$$\Lambda_{c+4,d_n+2,r_n+1} =  
\begin{pmatrix}
3^{d_n} & 0 \\
0 & 5^{r_n}
\end{pmatrix}
\begin{pmatrix}
7 \cdot 2^c & f \\
0 & 1	
\end{pmatrix}	\Z^2 = 
7 \cdot 2^c \cdot 3^{d_n}
\begin{pmatrix}
1 & f/(7 \cdot 2^c) \\
0 & \frac{5^{r_n}}{7 \cdot 2^c \cdot 3^{d_n}}
\end{pmatrix}	\Z^2.$$
To end the proof, it is sufficient to choose $r_n \in r_f+s_2\N$ and $d_n \in s_1 \N$ such that the following convergence holds.
\begin{eqnarray}
\frac{5^{r_n}}{7 \cdot 2^c \cdot 3^{d_n}} \xrightarrow[\; n \rightarrow + \infty]{} 1.
\end{eqnarray}
Set $p_n = d_n/s_1$ and $q_n = (r_n - r_f)/s_2$. Taking logarithms, we deduce the equivalent convergence in $(p_n,q_n)$:
$$ r_f \log (5) + q_n \cdot s_2 \log (5) - p_n \cdot s_1 \log(3) - \log(7 \cdot 2^c)  \xrightarrow[\; n \rightarrow + \infty]{} 0. $$
As in the proof of diagonal invariance, the forward orbit starting at $r_f \log(5) - \log(7 \cdot 2^c)$ for the translation $s_2 \log(5)$ in $\R / s_1\log(3)\Z$ is dense, so there exist sequences $p_n, q_n \rightarrow + \infty$ such that the above convergence holds.
	\end{proof}

\subsection{Density of shapes for periodic tori}
\label{ss:proof_of_density}

\begin{proof}[Proof of \cref{co:ds}]
For each $N\in\N$ with $N\geq 10$, we choose a pair of integers $(a_{1,N},a_{2,N})$ such that $a_{1,N},a_{2,N}-a_{1,N}>C$ for $C$ from \cref{le:where_is_shape} and 
		$$\left\lbrace \begin{array}{ cccccl}
			a_{1,N} &\equiv&0&\equiv& a_{2,N}-1 &\pmod{ 2^N}\\
			a_{1,N}-1 &\equiv&0&\equiv& a_{2,N} &\pmod{ 3^N}\\
			a_{1,N}-1 &\equiv&0&\equiv& a_{2,N}-1 &\pmod{ 5^N}
		\end{array} \right.$$

%Recall that $b_1=\alpha$ and $b_2=\alpha-a_1$. Applying \cref{le:where_is_shape} and enlarging $C$ if necessary, we know that up to the same scaling $\psi(b_1)=v_1$, $\psi(b_2)=v_2$ with $v_1,v_2$ given in \cref{le:where_is_shape}. Moreover, the matrix $g_N$ with $(\psi(b_1),\psi(b_2))=(u_1,u_2)g_N$ and its inverse $g_N^{-1}$ is bounded in $\GL_2(\R)^+$, where $(u_1,u_2)$ is a fixed oriented basis of $\frak a$ from its identification with $\R^2$. Therefore the lattice is represented by $[\Lambda(\calO^{(a_{1,N},a_{2,N})})]=[g_N\Z^2]$.
Recall that $b_1=\theta_0$ and $b_2=\theta_0-a_1$. Applying \cref{le:where_is_shape} and enlarging $C$ if necessary, we know that up to the same scaling $\psi(b_1)=v_1$, $\psi(b_2)=v_2$ with $v_1,v_2$ given in \cref{le:where_is_shape}. Moreover, the matrix $g_N$ with $(\psi(b_1),\psi(b_2))=(u_1,u_2)g_N$ and its inverse $g_N^{-1}$ is bounded in $\GL_2(\R)^+$, where $(u_1,u_2)$ is a fixed oriented basis of $\frak a$ from its identification with $\R^2$. Therefore the lattice is represented by $[\Lambda(\calO^{(a_{1,N},a_{2,N})})]=[g_N\Z^2]$.

From \cref{prop:suborder}, we know that for $c,d,r\leq N$
\[ \bigg[\Lambda\Big(\calO^{(a_{1,N},a_{2,N})}_{c,d,r}\Big)\bigg]=\bigg[\Big\{m\psi(b_1)+n\psi(b_2)\; \Big\vert \begin{pmatrix}
    m \\ n
\end{pmatrix}\in L_{c,d,r} \Big\}\bigg]=[g_NL_{c,d,r}]. \]
From \cref{defin-Lcde}, we further obtain that
\[[\Lambda(\calO^{(a_{1,N},a_{2,N})}_{c,d,r})]=[g_Ng\Lambda_{c,d,r}]=g_Ng[\Lambda_{c,d,r}] , \]
 where the action of $g_Ng$ is through $\mathrm{PGL(2,\R)}$.

Due to the uniform boundedness of $g_Ng$ for all $N$, we are able to choose a subsequence of $\{g_N g \}_{N \geq 10}$ that converges to some fixed $g_0\in \SL_2(\R)$, up to normalization by the determinant. Therefore, by \cref{theorm_density}, the set of shapes
\[ \bigcup_{N\in\N}\{[\Lambda(\calO^{(a_{1,N},a_{2,N})}_{c,d,r})],\ c,d,r\leq N \}%=\bigcup_{N\in\N}\{[g_Ng\Lambda_{c,d,r}],\ c,d,r\leq N \}
=\bigcup _{N\in\N}\{g_Ng[\Lambda_{c,d,r}],\  c,d,r\leq N \}  \]
is also dense in $\SL(2,\R)/\SL(2,\Z)$. 
\end{proof}

\subsection{Density of shapes of positive units}
\label{ss:positive_units}

As mentioned in \cref{re:order positive}-\ref{reit:positive_units}, the shapes of periodic tori of $A$-orbits in $\SL(3,\R)/\SL(3,\Z)$
form the set
$\{[\psi(\calO^{\times, +} ) ]\;|{\calO \text{ a totally real cubic order}}\}$. One can observe without much difficulty that for the totally positive units, in \cref{pr:proposition_polynom-units-orders} we would get 	
%		$$\mathcal{O}^{\times,+} = \langle \alpha, (\alpha - a_1)^2 \rangle.$$
		$$\mathcal{O}^{\times,+} = \langle \theta_0, (\theta_0 - a_1)^2 \rangle.$$
With this, we get that the totally positive units of the order $\calO_{c,d,r}$, in the notation of \cref{prop:suborder}, are given by 

$$\calO_{c,d,r}^{\times,+} =\left\lbrace b_1^m b_2^{2n} \; \bigg\vert  (m,n)  \in L_{c,d,r}' \right\rbrace,$$
where $L_{c,d,r}' \subset \Z^2$ is the lattice of pairs $(m,n) \in \Z^2$ satisfying the congruence conditions
$$\begin{array}{ cccl}
	m+2n &\equiv &0 &\pmod{ l_1 p_1^{\max( c-j_1,0)}},\\
	m-4n &\equiv &0 &\pmod{ l_2 p_2^{\max( d-j_2,0)}},\\
	2m-2n &\equiv &0 &\pmod{ l_3 p_3^{\max( r-j_3,0)}}.
\end{array}  $$

This leads to the following coincidence.
\begin{lem}
Let $\calO_{c,d,r}$ be as defined in \cref{ss:suborders-matrix-eq} with $(p_1,p_2,p_3)=(2,3,5)$.
Then the following shape points are equal in $\SL(2,\R)/\SL(2,\Z)$.
$$[\psi(\calO_{c,d,r}^{\times})]=[\psi(\calO_{c,d,r}^{\times, +} ) ].$$
\end{lem}
\begin{proof}
When we substitute $(p_1,p_2,p_3)=(2,3,5)$ as in the discussion at the beginning of \cref{ss:density}, we get the same values of $j_1,j_2,j_3,l_1,l_2,l_3$ as in the case of $L_{c,d,r}$ given in \cref{eq:defi_of_L}. Then, observe that $L_{c,d,r}$ and $L'_{c,d,r}$ differ only by the resubstitution $n \mapsto 2n$. This gives us 

$$ \left(
\begin{smallmatrix}
    1 & 0 \\
    0 & 2 
\end{smallmatrix}\right)
L_{c,d,r}' \subseteq L_{c,d,r}.$$

However, because $l_3 = 24$, the third condition in \cref{eq:defi_of_L} shows that $n$ must be even and in fact 
$$ \left(
\begin{smallmatrix}
    1 & 0 \\
    0 & 2 
\end{smallmatrix}\right)
L_{c,d,r}' = L_{c,d,r}.$$
In fact, what this means for $\calO^\times_{c,d,r}$ is that if $b_1,b_2$ are as in \cref{prop:suborder}, then $b_1^mb_2^n \in \calO^{\times}_{c,d,r} \Rightarrow b_1^m b_2^n \in \calO^{\times, +}_{c,d,r}$. Since the map $\psi:\calO^\times \rightarrow \frak a$ defined in \cref{eq:defi_of_psi} satisfies $\psi(b_1^mb_2^n)=\psi(-b_1^mb_2^n)$, we can conclude the statement of the lemma.
\end{proof}

With these facts, the proof of \cref{co:ds} adapts directly to prove the following statement.
\begin{theorem}
    The set of shapes of periodic tori for the left action of the subgroup of diagonal matrices with positive entries $A \subseteq \SL(3,\R)$ on $\SL(3,\R)/\SL(3,\Z)$ is dense in $\SL(2,\R)/\SL(2,\Z)$.
\end{theorem}

\section*{Competing interests} 
The authors have no competing interests to declare that are relevant to the content of this article.

		\raggedbottom
		\bibliographystyle{alpha}
	
	\bibliography{references}

@article{yifrach2026equidistribution,
	title = {The equidistribution of grids of rings of integers in number
	         fields of degrees 3, 4 and 5},
	author = {Yifrach, Y.},
	journal = {Combinatorics and Number Theory},
	volume = {15},
	number = {1},
	pages = {43--72},
	year = {2026},
	publisher = {Mathematical Sciences Publishers},
	issn = {2996-2196, 2996-220X},
	doi = {10.2140/cnt.2026.15.43},
}

@article{solan2023tori,
	title = {Tori approximation of families of diagonally invariant measures
	         },
	author = {Solan, O. N. and Yifrach, Y.},
	journal = {Geometric and Functional Analysis},
	volume = {33},
	number = {5},
	pages = {1354--1378},
	year = {2023},
	issn = {1016-443X, 1420-8970},
	doi = {10.1007/s00039-023-00646-7},
}

@article{bhargava2016equidistribution,
	title = {The equidistribution of lattice shapes of rings of integers in
	         cubic, quartic, and quintic number fields},
	author = {Bhargava, M. and Harron, P.},
	journal = {Compositio Mathematica},
	volume = {152},
	number = {6},
	pages = {1111--1120},
	year = {2016},
	publisher = {London Mathematical Society},
	doi = {10.1112/S0010437X16007260},
}

@misc{corso2025,
	title = {Unboundedness of shapes of unit lattices in totally real cubic
	         fields},
	author = {Corso, E. and Rodriguez Hertz, F.},
	year = {2025},
	eprint = {2502.16338},
	archiveprefix = {arXiv},
	primaryclass = {math.NT},
	url = {https://arxiv.org/abs/2502.16338},
	howpublished = {arXiv preprint \href{https://arxiv.org/abs/2502.16338}{arXiv:2502.16338}},
}

@article{shah_1996,
	author = {Shah, N. A.},
	title = {Limit distributions of expanding translates of certain orbits
	         on homogeneous spaces},
	journal = {Proceedings of the Indian Academy of Sciences -- Mathematical Sciences},
	volume = {106},
	number = {2},
	pages = {105--125},
	year = {1996},
	doi = {10.1007/BF02837164},
}

@misc{mathoverflow_disc,
	title = {{Answer to ``How random are unit lattices in number fields?''}},
	author = {Reznikov, A.},
	year = {2011},
	month = feb,
	howpublished = {\href{https://mathoverflow.net/a/56153}{MathOverflow}},
	url = {https://mathoverflow.net/a/56153},
	note = {Version dated 21 February 2011},
}

@phdthesis{sarnak1980,
	title = {Prime Geodesic Theorems},
	author = {Sarnak, P. C.},
	year = {1980},
	school = {Stanford University},
	address = {Stanford, CA},
	mrnumber = {2630950},
}

@article{knieper2005uniqueness,
	title = {The uniqueness of the maximal measure for geodesic flows on
	         symmetric spaces of higher rank},
	author = {Knieper, G.},
	journal = {Israel Journal of Mathematics},
	volume = {149},
	number = {1},
	pages = {171--183},
	year = {2005},
	doi = {10.1007/BF02772539},
}

@article{einsiedler_distribution_2011,
	title = {Distribution of periodic torus orbits and {Duke's} theorem for
	         cubic fields},
	volume = {173},
	issn = {0003-486X},
	doi = {10.4007/annals.2011.173.2.5},
	number = {2},
	journal = {Annals of Mathematics},
	author = {Einsiedler, M. and Lindenstrauss, E. and Michel, P.
	          and Venkatesh, A.},
	year = {2011},
	mrnumber = {2776363},
	pages = {815--885},
}

@article{huberZurAnalytischenTheorie1959,
	title = {{Zur analytischen Theorie hyperbolischer Raumformen und
	         Bewegungsgruppen}},
	volume = {138},
	issn = {0025-5831, 1432-1807},
	doi = {10.1007/BF01369663},
	journal = {Mathematische Annalen},
	author = {Huber, H.},
	year = {1959},
	mrnumber = {109212},
	pages = {1--26},
}

@article{margulisCertainApplicationsErgodic1969,
	title = {Applications of ergodic theory to the investigation of
	         manifolds of negative curvature},
	volume = {3},
	number = {4},
	journal = {Functional Analysis and Its Applications},
	author = {Margulis, G. A.},
	year = {1969},
	mrnumber = {0257933},
	pages = {335--336},
	issn = {0016-2663},
	doi = {10.1007/BF01076325},
	note = {English translation},
}

@article{deitmar,
	author = {Deitmar, A.},
	title = {A prime geodesic theorem for higher rank spaces},
	journal = {Geometric and Functional Analysis},
	volume = {14},
	year = {2004},
	number = {6},
	pages = {1238--1266},
	issn = {1016-443X, 1420-8970},
	mrnumber = {2135166},
	doi = {10.1007/s00039-004-0490-7},
}

@phdthesis{spatzier83,
	title = {Dynamical properties of algebraic systems: a study in closed
	         geodesics},
	author = {Spatzier, R. J.},
	year = {1983},
	school = {University of Warwick},
	url = {https://wrap.warwick.ac.uk/73525/},
}

@article{einsiedler_distribution_2009,
	title = {Distribution of periodic torus orbits on homogeneous spaces},
	volume = {148},
	issn = {0012-7094, 1547-7398},
	doi = {10.1215/00127094-2009-023},
	number = {1},
	journal = {Duke Mathematical Journal},
	author = {Einsiedler, M. and Lindenstrauss, E. and Michel, P.
	          and Venkatesh, A.},
	month = may,
	year = {2009},
	pages = {119--174},
	mrnumber = {2515103},
}

@article{bonthonneau_srb_2021,
	title = {{{SRB}} Measures for {{Anosov}} Actions},
	author = {Guedes Bonthonneau, Y. and Guillarmou, C. and Weich,
	          T.},
	year = {2024},
	month = nov,
	journal = {Journal of Differential Geometry},
	volume = {128},
	number = {3},
	pages = {959--1026},
	issn = {0022-040X},
	doi = {10.4310/jdg/1729092452},
}

@misc{gargava_density_of_shapes,
	author = {Gargava, N.},
	title = {{FormalShapes}},
	year = {2026},
	howpublished = {\href{https://github.com/nihargargava/DensityOfShapes/tree/3b4a50865dbab227f6ce96d703bd6dcf79d7c63e}{GitHub repository}},
	url = {https://github.com/nihargargava/DensityOfShapes/tree/3b4a50865dbab227f6ce96d703bd6dcf79d7c63e},
	note = {Commit 3b4a508, accessed 18 August 2026},
}

@incollection{clozelEQUIDISTRIBUTIONPOINTSHECKE2003,
	title = {{\'E}quidistribution des points de {Hecke}},
	author = {Clozel, L. and Ullmo, E.},
	booktitle = {Contributions to Automorphic Forms, Geometry, and Number Theory: A Volume in Honor of Joseph Shalika},
	editor = {Hida, Haruzo and Ramakrishnan, Dinakar and Shahidi, Freydoon},
	year = {2004},
	pages = {193--254},
	publisher = {Johns Hopkins University Press},
	address = {Baltimore, MD},
	isbn = {978-0-8018-7860-2},
	mrnumber = {2058609},
	language = {French},
}

@misc{conradUnit,
	title = {Dirichlet's unit theorem},
	author = {Conrad, K.},
	howpublished = {\href{https://kconrad.math.uconn.edu/blurbs/gradnumthy/unittheorem.pdf}{Expository notes available from the author's website}},
	url = {https://kconrad.math.uconn.edu/blurbs/gradnumthy/unittheorem.pdf},
	note = {Accessed 18 August 2026},
}

@inproceedings{cusickLower1983,
	title = {Lower Bounds for Regulators},
	booktitle = {Number Theory Noordwijkerhout 1983},
	author = {Cusick, T. W.},
	editor = {Jager, Hendrik},
	year = {1984},
	series = {Lecture Notes in Mathematics},
	volume = {1068},
	pages = {63--73},
	publisher = {Springer},
	address = {Berlin},
	isbn = {978-3-540-13356-8},
	doi = {10.1007/BFb0099441},
	mrnumber = {756083},
}

@article{davidShapesUnitLattices2017,
	title = {Dirichlet Shapes of Unit Lattices and Escape of Mass},
	author = {David, O. and Shapira, U.},
	year = {2018},
	month = may,
	journal = {International Mathematics Research Notices},
	volume = {2018},
	number = {9},
	eprint = {1607.04048},
	primaryclass = {math.NT},
	pages = {2810--2843},
	issn = {1073-7928, 1687-0247},
	doi = {10.1093/imrn/rnw324},
	archiveprefix = {arXiv},
}

@article{cusickRegulatorSpectrumTotally1991,
	title = {The Regulator Spectrum for Totally Real Cubic Fields},
	author = {Cusick, T. W.},
	year = {1991},
	month = sep,
	journal = {Monatshefte f{\"u}r Mathematik},
	volume = {112},
	number = {3},
	pages = {217--220},
	issn = {0026-9255, 1436-5081},
	doi = {10.1007/BF01297340},
	langid = {english},
}

@article{ankenyNoteClassNumbersAlgebraic1956,
	title = {A {{Note}} on the {{Class-Numbers}} of {{Algebraic Number
	         Fields}}},
	author = {Ankeny, N. C. and Brauer, R. and Chowla, S.},
	year = {1956},
	journal = {American Journal of Mathematics},
	volume = {78},
	number = {1},
	pages = {51--61},
	publisher = {Johns Hopkins University Press},
	issn = {0002-9327},
	doi = {10.2307/2372483},
}

@article{casselsProductInhomogeneousLinear1952,
	title = {The {{Product}} of n {{Inhomogeneous Linear Forms}} in n {{
	         Variables}}},
	author = {Cassels, J. W. S.},
	year = {1952},
	month = oct,
	journal = {Journal of the London Mathematical Society},
	volume = {s1-27},
	number = {4},
	pages = {485--492},
	issn = {0024-6107},
	doi = {10.1112/jlms/s1-27.4.485},
}

@article{remakUberGrossenbeziehungenZwischen1952,
	title = {{{\"U}ber Gr{\"o}{\ss}enbeziehungen zwischen Diskriminante und
	         Regulator eines algebraischen Zahlk{\"o}rpers}},
	author = {Remak, R.},
	year = {1952},
	journal = {Compositio Mathematica},
	volume = {10},
	pages = {245--285},
	issn = {0010-437X},
	url = {https://www.numdam.org/item/CM_1952__10__245_0/},
	mrnumber = {54641},
	zbl = {0047.27202},
	langid = {deu},
}

@article{silvermanInequalityRelatingRegulator1984,
	title = {An Inequality Relating the Regulator and the Discriminant of a
	         Number Field},
	author = {Silverman, J. H.},
	year = {1984},
	month = dec,
	journal = {Journal of Number Theory},
	volume = {19},
	number = {3},
	pages = {437--442},
	issn = {0022-314X},
	doi = {10.1016/0022-314X(84)90082-9},
	langid = {english},
}

@incollection{kleinbock1996bounded,
	title = {Bounded Orbits of Nonquasiunipotent Flows on Homogeneous Spaces},
	author = {Kleinbock, D. Y. and Margulis, G. A.},
	booktitle = {Sinai's Moscow Seminar on Dynamical Systems},
	editor = {Bunimovich, L. A. and Gurevich, B. M. and Pesin, Ya. B.},
	series = {American Mathematical Society Translations, Series 2},
	volume = {171},
	pages = {141--172},
	year = {1996},
	publisher = {American Mathematical Society},
	address = {Providence, RI},
	isbn = {978-0-8218-0456-8},
	doi = {10.1090/trans2/171/11},
	mrnumber = {1359098},
}

@article{dangEquidistributionCountingPeriodic2023,
	title = {Equidistribution and Counting of Periodic Tori in the Space of
	         {{Weyl}} Chambers},
	author = {Dang, N.-T. and Li, J.},
	year = {2026},
	journal = {Commentarii Mathematici Helvetici},
	volume = {101},
	number = {1},
	pages = {47--113},
	issn = {0010-2571, 1420-8946},
	eprint = {2305.17070},
	primaryclass = {math.DS},
	doi = {10.4171/CMH/594},
	archiveprefix = {arXiv},
}
	
	\bigskip  

	  N.T.~Dang, \textsc{Institut de Mathématique d'Orsay,
    Paris}\par\nopagebreak
  \textit{E-mail address}:  \texttt{nguyen-thi.dang@universite-paris-saclay.fr }

  \medskip

  N.~Gargava, \textsc{Institut de Mathématique d'Orsay,
    Paris}\par\nopagebreak
  \textit{E-mail address}:  \texttt{nihar.gargava@universite-paris-saclay.fr}

  \medskip

  J.~Li, \textsc{%CNRS and CMLS, École Polytechnique, Paris. Current address:
    School of Mathematical Sciences, Fudan University, Shanghai}\par\nopagebreak 
  \textit{E-mail address}:  \texttt{jialunli@fudan.edu.cn}

\end{document}